\documentclass[11pt]{amsart}
\pdfoutput=1
\usepackage{tikz}
\usetikzlibrary{arrows,shapes,positioning,calc,patterns}
\usepackage[mathscr]{eucal}
\usepackage{mathtools}
\usepackage[colorlinks=true,linkcolor=blue,urlcolor=blue, citecolor=blue]{hyperref}
\usepackage[T1]{fontenc}
\usepackage{mathptmx}
\usepackage{microtype}
\usepackage{amssymb}
\usepackage{ MnSymbol}
\usepackage{graphicx}
\usepackage{color}
\usepackage{latexsym}
\usepackage{amssymb,amsmath,amstext}
\usepackage{epsfig}
\usepackage{graphics}
\usepackage{amsthm}
\usepackage{verbatim}
\usepackage{placeins}
\usepackage{tikz-cd} 
\usepackage{hyperref}

\numberwithin{equation}{section}

\usepackage[centering, includeheadfoot, hmargin=1in, vmargin=1in,
  headheight=30.4pt]{geometry}
\newtheorem{lemma}{Lemma}[section]
\newtheorem{theorem}[lemma]{Theorem}

\newtheorem{corollary}[lemma]{Corollary}
\newtheorem{proposition}[lemma]{Proposition}

\newtheorem{conjecture}[lemma]{Conjecture}
\newtheorem{definition}[lemma]{Definition}
\theoremstyle{theorem}
\newtheorem{remark}[lemma]{Remark}
\newtheorem{example}[lemma]{Example}

\theoremstyle{definition}
\newtheorem*{defn*}{Definition}
\theoremstyle{plain}
\newtheorem*{thm*}{Theorem}
\theoremstyle{plain}
\newtheorem*{prop*}{Proposition}
\theoremstyle{plain}
\newtheorem*{conj*}{Conjecture}

\newcommand{\C}{\mathbb{C}}
\newcommand{\Z}{\mathbb{Z}}
\newcommand{\PP}{\mathbb{P}}
\newcommand{\R}{\mathbb{R}}

\newcommand{\codim}{\text{\rm codim}}

\allowdisplaybreaks

\begin{document}

\title[Iterated and mixed discriminants]{\bf  Iterated and mixed discriminants}

\author{Alicia Dickenstein, Sandra di Rocco, Ralph Morrison
}

\address[Alicia Dickenstein]{Department of Mathematics, FCEN, University of Buenos Aires and IMAS (UBA-CONICET),
Ciudad Universitaria, Pab. I, C1428EGA Buenos Aires, Argentina}
\email{alidick@dm.uba.ar}
\address[Sandra Di Rocco] {KTH, Royal Institute of Technology, 10044, Stockholm, Sweden}
\email{dirocco@kth.se}
\address[Ralph Morrison]{Department of Mathematics, Williams College, Williamstown, MA 01267, USA}
\email{10rem@williams.edu}

\date{}

  \thanks{The first author acknowledges the support of
UBACYT 20020170100048BA and
CONICET PIP 11220200100182, Argentina. 
The second author acknowledges supported by KTH and Williams College. 
The third author acknowledges support by ICERM and VR grants NT:2014-4763, NT:2018-03688. 
All three authors acknowledge support by the Knut and Alice Wallenberg foundation.} 

\begin{abstract}\noindent 
Classical work by Salmon and Bromwich classified singular intersections of two quadric surfaces. The basic idea of these results was already pursued by Cayley in connection with tangent intersections of conics in the plane and used by Schäfli for the study of hyperdeterminants. More recently, the problem has been revisited with similar tools in the context of geometric modeling and a generalization to the case of two higher dimensional quadric hypersurfaces was given by Ottaviani.  We propose and study a generalization of this question for systems of Laurent polynomials with support on a fixed point configuration. 

In the non-defective case, the closure of the locus of coefficients giving a non-degenerate multiple root  of the system is defined by a polynomial called the {\em mixed discriminant}.  
We define a related polynomial called the multivariate {\em iterated discriminant}, generalizing the classical Sch\"afli method for hyperdeterminants. 
This iterated discriminant is easier to compute and we prove that it is always divisible by the mixed discriminant. We show that tangent intersections 
can be computed via iteration if and only if the singular locus of a corresponding dual variety has sufficiently high codimension.  
We also study when point configurations corresponding to Segre-Veronese varieties and to the lattice points of planar smooth polygons,
have their iterated discriminant equal to their mixed discriminant.

\end{abstract}

\maketitle
\section{Introduction}

Let $K$ be an algebraically closed field  of characteristic zero and  $A\subset\Z^n$ a finite lattice subset. 
A (Laurent) polynomial $p=\sum_{a\in A} c_a x^{a}\in K[x_1,\ldots,x_n]$ with support on the point configuration $A$ is called an \emph{$A$-polynomial}.

Classical work by Salmon \cite{Sal2} and Bromwich \cite{Bromwich}  classified singular intersections of two quadric surfaces, corresponding 
to the case of two $A$-polynomials where $A$ consists of the lattice points in the dilated  simplex $2 \Delta_3$ in $\R^3$. 
The basic idea of these results was already pursued by Cayley in connection with tangent intersections of conics in $\C^2.$
More recently,
the problem has been revisited with similar tools in~\cite{FNO}, in the context of geometric modeling with  focus on the real case; and  in~\cite{ZM19}, where these
 techniques are used to classify singular Darboux cyclides.  These are surfaces in 3-space that are the projection of the intersection
 of two quadrics in dimension four. 
A generalization to the case of two higher dimensional quadric hypersurfaces is given in \cite{O}.

Consider two space quadrics, given in matrix form by  \begin{equation}\label{eq:quadric}
p_i= \begin{bmatrix} 1&x_1&x_2&x_3 \end{bmatrix} \,  M_i \,  \begin{bmatrix}1\\ x_1\\x_2\\x_3\end{bmatrix},\quad i=0,1.
\end{equation}
 For generic matrices $M_i \in K^{4 \times 4}$, the intersection
$(p_1=p_2=0)$ describes a non-singular curve of degree $4$. The non-generic intersections are described in \cite{Sch, GKZ} in the following way. 
Consider the pencil of quadrics given by  $p_0+t p_1$. 
Using  the Sch\"afli decomposition method, the existence of a tangential intersection can be studied by
considering the zero locus of the following polynomial in the entries of $M_0, M_1$:
\begin{equation}\label{hyper}D_{4 \Delta_1}(\det(M_0+tM_1)),\end{equation}
where $D_{4,\Delta_1}$ is the univariate discriminant of the degree $4$ polynomial 
 $\det(M_0+tM_1)$, considered as a polynomial in $t.$ For generic matrices this is a polynomial of degree $6$ in
its entries (that is, in the coefficients of $p_0, p_1$), and it vanishes whenever $\det(M_0+tM_1)$  does not have four simple
roots.  To classify the different singular intersections, they then studied the Segre characteristics arising from the Jordan
normal form of the matrix $M_0+tM_1$.

\begin{figure}\label{fig:hyper}
\includegraphics[scale=0.35]{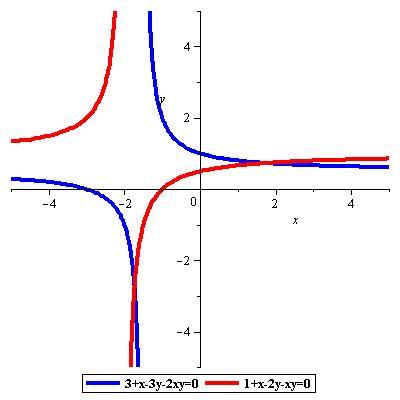}
\qquad \qquad
\includegraphics[scale=0.35]{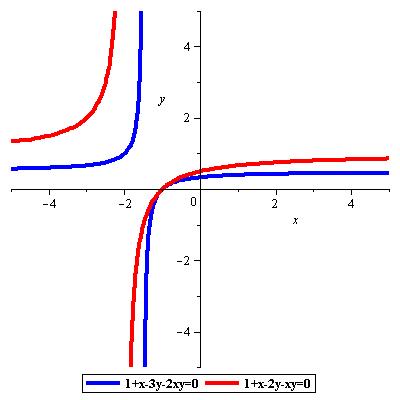}
\caption{Transverse (left) and non-transverse (right)  hyperbolas}
\end{figure}

In this paper, we propose and study a generalization of this approach for any support $A$. 

We consider Equation~\eqref{hyper} to be an {\em iterated process}, as we are computing the discriminant of a discriminant. 
Factorizations of iterated discriminants and resultants for
polynomials of three variables where studied in~\cite{BM}. 
Our aim is to define and study an {\em iterated discriminant} generalizing Sch\"afli's method for hyperdeterminants, and to show when tangent intersections can be computed via iteration. 

The theory of $A$-discriminants was introduced in  \cite{GKZ} and has been extensively studied both from a geometric and a 
computational viewpoint  \cite{DFS,DRRS,E10,GHRS}.  Denote by  $X_A \subset {\mathbb P^{|A|-1}}$ 
  the projective variety defined as the closed image of the monomial embedding given by the $A$-monomials. 
  The dual variety $X_A^\nu \subset{\mathbb P^{|A|-1}}^\nu$ is the closure of the coefficient vectors of the $A$-polynomials $p$ 
  whose zero-locus $(p=0)$ has a singular point $x\in (K^*)^n$ with nonzero coordinates.
   Equivalently, the dual variety is the  closure  of the hyperplane 
sections of  $X_A$ which are singular at a point with nonzero coordinates.
The expected codimension of $X_A^\nu$ is one and when this is the case we say that $A$ is \emph{non-defective}. 
When $A$ is non-defective  the irreducible polynomial $D_A \in {\mathbb Z}[(c_a)_{a \in A}]$ defining 
(up to sign) the dual variety: $X_A^\nu= (D_A=0),$ is called the \emph{$A$-discriminant} ~\cite{GKZ}. 
We will use the notation $D_A((c_a)_{a \in A})=D_A(p)$.

 If $n=1$ and $A=\{0,1,2\}$, then
$p=c_2 x^2 + c_1x+c_0$ and $D_A(p)= c_1^2 - 4 c_0 c_2$ is the classical discriminant of a degree two
polynomial. More generally,  $D_{\{0, 1,\ldots, \delta\}}$ coincides with the classical discriminant of univariate
polynomials of degree $\delta$. This $D_{0, 1,\ldots, \delta}$ is a polynomial of degree $2(\delta-1)$ in the coefficients,
which we denote by $D_{\delta \Delta_1}$. The case of multi-linear polynomials (i.e. tensors)
corresponds to the case in which the convex hull of $A$ equals the product $\Delta_{n_1}\times\ldots\times \Delta_{n_l}$ 
where $\Delta_s$ denotes the unit simplex of dimension $s.$ This multivariate $A$-discriminant is also referred to as the hyperdeterminant of size 
$(n_1+1)\times\ldots\times(n_l+1)$ \cite[Chapter 14]{GKZ}. 
This is a classical object defined originally by Cayley~\cite{Cayley}. 

Note that for a quadratic polynomial $p$ with associated matrix $M$ as in~\eqref{eq:quadric}, that is for $A$ consisting of the lattice points in $2 \Delta_3$,
the existence of a singular point in $(p=0)$ implies that the linear forms given by its
partial derivatives vanish and so $\det(M)=0$. Indeed, $D_A(p) = \det(M)$ (up to an integer factor). This suggests that an iterated 
discriminant should be connected to the notion of discriminant for a system of polynomials. 
This notion is called the \emph{mixed discriminant}~\cite{GKZ,CCDRS,DEK}, which is a natural generalization of the classical $A$-discriminant.

 Given $r+1$ finite configurations $A_0, \ldots, A_r \subset \Z^n$, and a system of $A_i$-polynomials $p_0, \dots, p_r$
 \begin{equation}\label{system} p_{0}=p_{1}=\ldots=p_{r}=0, \quad  p_i=\sum_{a\in A_i} c_{i,a} x^{a},\end{equation}
  we call an isolated solution $x\in(K^*)^n$ a \emph{non-degenerate multiple root} for the system \eqref{system} if the $r+1$ 
gradient vectors $\nabla_xp_i(x), i=0,\ldots,r$ are linearly dependent but any
subset of $r$ of them is linearly independent.  
The associated \emph{mixed discriminantal variety} is the closure of the locus of coefficients for which the system 
has a non-degenerate multiple root.  If this variety is a hypersurface, it is defined by a single irreducible 
polynomial which we call the mixed discriminant, denoted $MD_{A_0,\ldots A_{r}}$.  
If it is not a hypersurface, we call the system \emph{defective} and set $MD_{A_0,\ldots A_{r}}=1$.

Observe that when $r=0$, $MD_{0,A_0} = D_{A_0}$ equals the $A_0$ discriminant. In fact, in the non-defective case, mixed
discriminants are special cases of discriminants of a single polynomial. This was settled in~\cite{GKZ}, but without the hypothesis of non-degeneracy of the
common multiple root and in~\cite{CCDRS} for the case $r+1=n$.
Given $A_0,\ldots A_{r}$, the associated Cayley configuration $C=C(A_0, \dots, A_r) \subset \mathbb{Z}^{n+r}$ is the union
of the lifted configurations $e_i\times A_i\in \mathbb{Z}^{n+r}$ for $i=0,\ldots,r$, where $e_0=0$ and $e_i$ is the standard $i^{th}$ basis vector in
$\mathbb{Z}^r$ for $i\geq 1.$  As sparse discriminants are affine invariants of lattice configurations~\cite{GKZ},
we could equivalently consider $C \subset  \mathbb{Z}^{n+r+1}$, where now $e_0, \dots, e_r$ denote the
canonical basis in $\Z^{r+1}$.
We introduce $r+1$ new variables $\lambda_0,\ldots,\lambda_r$ and encode the initial system by one auxiliary $C$-polynomial:
$$P_\lambda=\lambda_0p_0+\ldots+\lambda_rp_r \in K[\lambda_0, \dots, \lambda_r, x_1, \dots, x_n].$$ 
We will denote both this polynomial and its tuple of coefficients by $P_{\lambda}$ where $\lambda=(\lambda_0,\ldots,\lambda_r).$ 
In Proposition~\ref{prop} we prove that when $C$ is non-defective, 
$MD(A_0,\ldots, A_r)(p_0,\dots, p_r)$ can be computed as $D_C(P_\lambda)$ for any $r$.

 This characterization leads to the following definition of {\em multivariate iterated discriminant} of order $r$.
In the present paper we consider the case when  $A_0=\ldots=A_r=A$ and use the  notation $MD_{r,A}:=MD_{A,\ldots,A}.$ Notice that $D_A(P_\lambda)$ is a 
homogeneous polynomial of degree $\deg(D_A)$ in $\lambda_0,\ldots,\lambda_r.$ 

\begin{definition} \label{def:id}
Given $A \subset \Z^n$ non-defective, denote by $d$ the codimension of the singular locus of the
dual variety $X_A^\nu$. Given $r \ge 0$, 
the \emph{ multivariate iterated discriminant of order $r$} is the polynomial $ID_{r,A}$ on the
coefficients of $(r+1)$ $A$-polynomials $p_0, \dots, p_r$ defined by 
$$\begin{cases} 
 ID_{r,A}(p_0, \dots, p_r):=D_{\delta_A\Delta_{r}}(D_A(P_\lambda)), \, \text{ if } d \ge r,\\
 ID_{r,A}(p_0, \dots, p_r):=0, \, \text{ otherwise. }
\end{cases}$$
\end{definition}
It is worth noting that in the classical case of $r=0,$ all these polynomials coincide by definition: $$MD_{0,A}=ID_{0,A}=D_A, \;\; \text{ and } D_{\delta\Delta_{0}}(D_A(\lambda p_A))=D_A. $$ 
The latter equality is a consequence of the fact that the discriminant (in the variable $\lambda$) of the monomial 
$D_A\lambda^{\delta}$ is the coefficient $D_A$ \cite{Jou}. Moreover, when $A$ consists of the vertices of a simplex, $ID_{r,A}$ 
coincides with the hyperdeterminant Sch\"afli decomposition \cite[Ch. 14]{GKZ}.

\medskip

Our main results give a precise relation between $MD_{r,A}$ and $ID_{r,A}$. 
The advantage of relating $MD_{r,A}$ with $ID_{r,A}$ is that the latter polynomial is much easier to compute.
We show that in the non-defective case $MD_{r,A}$ is always an irreducible factor of $ID_{r,A}$, as a consequence of  biduality  (see Section~\ref{section:iterated}). Therefore,
if $ID_{r,A}(p_0, \dots, p_r) \neq 0$, we get a certificate that
the intersection $(p_0=\dots=p_r)$ is smooth. 
When $A$ is non-defective, we denote by $\textrm{sing}(X_A^\nu)$ the subscheme of $X_A^\nu$ defined by
 the ideal generated by the partial derivatives of $D_A.$ We show that $ID_{r,A}$ can have other irreducible
factors given by the Chow forms $Ch_{Y_k}$ of the higher dimensional irreducible components of the schematic singular locus of the dual variety $X_A^\nu$.  
We recall the notion of Chow forms at the beginning of Section~\ref{section:iterated}.  Theorem~\ref{th:main} and Proposition~\ref{discr} imply the following Theorem.

\begin{thm*} Assume $A \subset \Z^n$ is non-defective and let $r\in\Z$ with $0 \le r \le \dim(X_A).$ 
Then,  the mixed discriminant $MD_{r,A}$  always divides the iterated discriminant $ID_{r,A}$. Moreover:
\begin{enumerate}
\item If ${\rm codim}_{X_A^\nu}(sing(X_A^\nu))> r$, then $ID_{r,A} = MD_{r,A}$.
\item If ${\rm codim}_{X_A^\nu}(sing(X_A^\nu))= r$, then $ID_{r,A} = MD_{r,A}  \prod_{i=k}^\ell Ch_{Y_k}^{\mu_k},$
where $Y_1, \dots, Y_\ell$ are the irreducible components of $sing(X_A^\nu)$ of codimension $r$, with respective multiplicities $\mu_k\ge 2$.
\item If ${\rm codim}_{X_A^\nu}(sing(X_A^\nu)) < r$,  
then  $ID_{r,A}=0.$
\end{enumerate}
\end{thm*}

The paper is organized as follows. In Section \ref{section:motivating} we present some examples that motivate the theory of iterated discriminants.  
In Section \ref{section:mixed} we present material on 
mixed discriminants and Cayley configurations. 

In Section \ref{section:iterated} we develop the theory of iterated discriminants and prove our main results
Theorem~\ref{th:main} and Proposition~\ref{discr}.
 We prove in Proposition~\ref{prop:muk} that the multiplicities $\mu_k$ in Theorem~\ref{th:main} are at least two. Very few is known 
in general about these multiplicities, except for the homogeneous case of three variables studied in~\cite{BM} and the general results in~\cite{Lazard}. 
Based on this evidence and some examples we computed, we state Conjecture~\ref{conj:muk}. The difficulty in
determining these multiplicities relies in the fact that for general point configurations $A$, a complete description of the components of
 the singular locus of the dual varieties $X_A$  and their codimensions is out of reach for the moment. By a result of Katz (Prop.~ 3.4 in~\cite{Katz}) it is expected that the codimension one
components correspond to the double point locus (the closure of those hypersurfaces with two different non-degenerate
singular points) and the cusp locus (the closure of those hypersurfaces having a single degenerate singular point with an $A_2$-singularity).  
The case of hyperdeterminants has been exhaustively described in Theorem 0.5 in~\cite{WZ}, where it is shown that in the non-defective case only one irreducible component 
with codimension one can exist, or there could be several irreducible components of codimension one of both types.  Already the univariate {\sl sparse} case poses some challenges~\cite{DHT}. 
Even the particular case of the existence of a cusp component with codimension one when $D_A$ corresponds to the mixed discriminant of two planar configurations, recently studied in~\cite{Ni}, is not trivial.
A general approach to
describe the irreducible components (and much more information) via the computation of tropical fans and characteristic classes is developed in~\cite{E18}.  

In Section~\ref{section:comparing} we ask more broadly when mixed and iterated discriminants are equal, for products of scaled simplices, that is, when $X_A$ is a Segre-Veronese variety. 
The case of Segre varieties was solved in~\cite{WZ}, via a careful study of the
singularities of hyperdeterminant varieties.  
 As a corollary of our results, we show in Proposition~\ref{proposition:r1d2}  that the iterated  method to characterize singular complete intersections 
for $r+1$ hypersurfaces of the same degree $d>1$ in $\mathbb P^n$ gives the corresponding mixed discriminant if and only if
$r=1$ and $d=2$ (the case of two quadric hypersurfaces already found in~\cite[Theorem 8.2]{O}).

Our 
Conjecture~\ref{conjecture:main} is the following, with notation as in Section \ref{section:comparing}:
\begin{conj*}
 The equality $\deg(ID_{r,A_{\ell,d,k}})=\deg(MD_{r,A_{\ell,d,k}})$ holds if and only if $$\mathbb{P}^{r}(1)\times\mathbb{P}^{k_1}(d_1)
\times\cdots\times \mathbb{P}^{k_\ell}(d_\ell)$$ is of one of the following cases:
 \begin{enumerate}
 \item $\PP^r\times\PP^m\times\PP^m,\, m\geq 1,\, r=1,2$,
 \item $(\PP^1)^4$,
 \item $\PP^1\times\PP^n(2)$.
 \end{enumerate}
 \end{conj*} 
A partial answer is given in Theorem \ref{theorem:d_i>1} and Proposition~\ref{proposition:r1d2}.  

Finally, in Section \ref{section:plane} we analyze the case of plane curves. Theorem~\ref{th:plane} shows that for planar configurations $A$ consisting of the lattice points of
a smooth polygon, the only case where $MD_{1,A}$ equals $ID_{1,A}$ are the known cases in which the polygon is
the unit square (the bilinear case) or $2 \Delta_2$, the standard triangle of size $2$. This implies that
in all other cases, the singularities of the discriminant locus have codimension one; that is, there are ``many'' different
types of singular hypersurfaces defined by $A$-polynomials. A factorization of the iterated discriminants  gives
all components of the singular locus of codimension one.

\subsection*{Acknowledgements}   We thank Carlos D'Andrea, Fr\'{e}d\'{e}ric Bihan, Laurent Bus\'{e},  Bernard Mourrain, and Giorgio Ottaviani for helpful discussions and references
 to previous work in this direction.

\section{Motivating examples}\label{section:motivating}\label{example}

 In this section we present some motivating examples that we abstract in the paper.  
 The first two correspond to two classical cases in which the iterated discriminant actually computes the mixed
 discriminant. The last two are the simplest cases which already show the occurrence of other factors of the
 iterated discriminant.

 \begin{example} \label{ex:1x1x1}
{\rm Let $A=\{(0,0),(1,0),(0,1),(1,1)\}$ be the vertices of the unit cube and 
let $f = c_{00}+c_{10}x_1 +c_{01}x_2+c_{11} x_1 x_2$ be an $A$-polynomial. 
In this case, $D_A (f)=c_{00}c_{11}-c_{10}c_{01} $ is a polynomial of degree $2$,
which equals the determinant of the matrix
\[\left(\begin{array}{cc}
c_{00} & c_{01}\\
c_{10} & c_{11}
\end{array}\right).\]

In case $r+1=2$, the mixed discriminant associated with two $A$-polynomials $p_0= c^1_{00}
+c^1_{10}x_1 +c^1_{01}x_2+c^1_{11} x_1 x_2$ and $p_1 = c^2_{00}+c^2_{10}x_1 +c^2_{01}x_2+c^2_{11} x_1 x_2$, is the following degree four irreducible polynomial,
which is the hyperdeterminant of format $2 \times 2 \times 2$ (see \cite{GKZ}, pp. 475--479):
\[
MD_{1,A}(p_0,p_1) ={c^2_{00}}^2 {c^1_{11}}^2 - 2 c^2_{00} c^2_{01} c^1_{10} c^1_{11} - 2 c^2_{00} c^2_{10} c^1_{01} c^1_{11} -2 c^2_{00} c^2_{11} c^1_{00} c^1_{11} +\]
\[ 4 c^2_{00} c^2_{11} c^1_{01} c^1_{10} + 2 c^2_{00} c^1_{00} {c^1_{11}}^2 - 4 c^2_{00} c^1_{01} c^1_{10} c^1_{11} + {c^2_{01}}^2 {c^1_{10}}^2 + 4 c^2_{01} c^2_{10} c^1_{00} c^1_{11} - \]
\[2 c^2_{01} c^2_{10} c^1_{01} c^1_{10} - 2 c^2_{01} c^2_{11} c^1_{00} c^1_{10} + 2 c^2_{01} c^1_{00} c^1_{10} c_{11} + {c^2_{10}}^2 {c^1_{01}}^2 - \]
\[2 c^2_{10} c^2_{11} c^1_{00} c^1_{01}+2 c^2_{10} c^1_{00} c^1_{01} c^1_{11} + {c^2_{11}}^2 {c^1_{00}}^2- 2 c^2_{11} {c^1_{00}}^2 c^1_{11} + 
{c^1_{00}}^2 {c^1_{11}}^2,
\]
It vanishes at $(p_0,p_1)$ with respective coefficient vectors $(1,1,-2,-1)$ and $(1,1,-3,-2)$, corresponding to the tangent hyperbolas in Figure~\ref{fig:hyper}.

One form of computing $D_A$ is as the iterated discriminant $ID_{1,A}$. 
Write
\[ \det \left(\begin{array}{cc}
c^1_{00} + \lambda c^2_{00} & c^1_{01} + \lambda c^2_{01}\\
c^1_{10} +\lambda c^2_{10} & c^1_{11} + \lambda c^2_{11}
\end{array}\right) \, = \, \Delta_0 + \Delta_1 \lambda + \Delta_2 \lambda^2,\]
and then compute
\[ MD_{1,A} (c^1, c^2) =  \Delta_1^2 - 4 \Delta_0 \Delta_2, \]
as the univariate resultant of the degree $2$ polynomial $\Delta_0 + \Delta_1 \lambda + \Delta_2 \lambda^2$ in $\lambda$ with coefficients in $\Z[c^1, c^2]$.
This compact formula is the simplest case of Sch\"afli's formula to compute the mixed discriminant $MD_{1,A}$. 
}
\end{example}

\begin{example}\label{ex:Farouki}
{\rm Let us consider again the case discussed in the Introduction corresponding to the singular intersections of two quadric surfaces $p_0, p_1$ in three-space.  
We display their common support $A$ as the columns of the following $3 \times 10$ matrix:
\[\begin{bmatrix}
0&1&0&0&2&1&1&0&0&0\\
0&0&1&0&0&1&0&2&1&0\\
0&0&0&1&0&0&1&0&1&2
\end{bmatrix}.\]
We also display the corresponding Cayley configuration $C= \Delta_1 \times A$ as the columns of the following $5 \times 20$-matrix:
\setcounter{MaxMatrixCols}{30}
$$
 \begin{small}
\begin{bmatrix}
0&1&0&0&2&1&1&0&0&0&0&1&0&0&2&1&1&0&0&0\\
0&0&1&0&0&1&0&2&1&0&0&0&1&0&0&1&0&2&1&0\\
0&0&0&1&0&0&1&0&1&2&0&0&0&1&0&0&1&0&1&2\\
1&1&1&1&1&1&1&1&1&1&0&0&0&0&0&0&0&0&0&0\\
0&0&0&0&0&0&0&0&0&0&1&1&1&1&1&1&1&1&1&1
\end{bmatrix}
\end{small}.
$$
In this case, we know that $(X_C)^\nu$ is a hypersurface by~\cite{DiR}. 
Thus we have that the polynomial $MD_{1,A}(p_0,p_1)$ cuts out the closure of the locus of coefficients for which the two quadrics lie tangent to one another at a point
and it can be computed via the discriminant $D_C$ by Proposition~\ref{prop}. It can be also computed as the iterated discriminant in~\eqref{hyper}. 
This polynomial can be studied through tropical discriminants as in~\cite{DFS}.  

Moreover,  one can compute the univariate discriminant $D_{4 \Delta_1}$ of a degree $4$ polynomial as the discriminant of its cubic resolvent
from Galois theory.
Let $\Delta_i$ denote the coefficient $\lambda^i$ in $\det(M_0+tM_1)$.  
Then
$$MD_{1,A} (p_0,p_1)=\frac{4p^3-q^2}{27},$$
where
$p=12\Delta_4\Delta_0-3\Delta_3\Delta_1+\Delta_2^2$, and
$q=72\Delta_4\Delta_2\Delta_0+9\Delta_3\Delta_2\Delta_1-27\Delta_4\Delta_1^2-27\Delta_0\Delta_3^2-2\Delta_2^3$.

This gives a compact and feasible way of computing the mixed discriminant $MD_{1,A}$ we are interested in. In fact,
expanding this expression in terms of
the coefficients of $p_0,p_1$ is beyond the capabilities of the excellent Computer Algebra System Macaulay2~\cite{M2} in a standard computer, because
it is a polynomial of degree $24$ which has degree $12$ in both the coefficients of $p_0$ and $p_1$. 
Note that a general polynomial of bidegree $(12,12)$  in two groups of $10$ variables has more than  $4 \cdot 10^{11}$ monomials! 
}

\end{example}

The general case is hinted in the following simple examples.

\begin{example} \label{ex:12}
{\rm Consider the two dimensional configuration $A = \{ (0,0), (1,0), (2,0), (0,1), (1,1)\}$
 corresponding to the first Hirzebruch surface ${\mathbb F}_1$.
\definecolor{zzttqq}{rgb}{0.6,0.2,0.}
\definecolor{qqqqff}{rgb}{0.,0.,1.}
\definecolor{xdxdff}{rgb}{0.490196078431,0.490196078431,1.}
\begin{center}\begin{tikzpicture}[line cap=round,line join=round,>=triangle 45,x=1.0cm,y=1.0cm]
\fill[color=zzttqq,fill=zzttqq,fill opacity=0.1] (0.,2.) -- (0.,0.) -- (4.,0.) -- (2.,2.) -- cycle;
\draw [color=zzttqq] (0.,2.)-- (0.,0.);
\draw [color=zzttqq] (0.,0.)-- (4.,0.);
\draw [color=zzttqq] (4.,0.)-- (2.,2.);
\draw [color=zzttqq] (2.,2.)-- (0.,2.);
\begin{scriptsize}
\draw [fill=xdxdff] (2.,0.) circle (1.5pt);
\draw [fill=xdxdff] (4.,0.) circle (1.5pt);
\draw [fill=xdxdff] (0.,2.) circle (1.5pt);
\draw [fill=qqqqff] (2.,2.) circle (1.5pt);
\draw [fill=xdxdff] (0.,0.) circle (1.5pt);
\end{scriptsize}
\end{tikzpicture}
\end{center}

Given a generic polynomial $f$ with support $A$:
\[ f(x, y) = a_0 + a_1 x + a_2 x^2 + y (b_0 + b_1 x),\]
the $A$-discriminant coincides with the resultant of the two univariate polynomials
$a_0 + a_1 x + a_2 x^2$ and $b_0 + b_1 x$ and thus is equal to the
degree $3$ polynomial
\[D_A(f) \, = a_0 b_1 ^2 - a_1 b_0 b_1 + a_2 b_0^2.\]
The mixed discriminant $MD_{1,A}$ has degree $8$, while the iterated discriminant $ID_{2,A}$ has degree $2 \cdot 3\cdot (3-1)^{1}  = 12$ by~\eqref{deg} .
There is another irreducible factor  that we explain in Theorem~\ref{th:main} and compute in Example~\ref{ex:12c}.} 
\end{example}

\begin{example}\label{ex:deg3}
{\rm
We now consider the case of a univariate polynomial of degree $3$ with $A=\{0,1,2,3\}$ and $r=1$.
Given two cubic polynomials $p_0, p_1$ depending on a variable $x$, their mixed discriminant equals the discriminant
of the Cayley configuration $$C=\{(0,0), (0,1), (0,2), (0,3), (1,0), (1,1), (1,2), (1,3)\}$$ at the polynomial
$p_0 + t p_1$ in one more variable $t$.  In fact, $D_C(p_0+tp_1)$ equals the univariate resultant ${\rm Res}_{3,3}(p_0, p_1)$.
This resultant can be computed as the determinant of the associated Sylvester matrix and therefore has degree $6$
in the vectors of coefficients of $p_0,p_1$. Since the discriminant $D_A$ of a cubic univariate polynomial has degree $4$,  
the iterated discriminant $ID_{1,A} =D_{4 \Delta_1}(D_A(p_0+tp_1))$ instead has degree $ 2 \cdot 4 \cdot 3 = 24$ according to~\eqref{deg}. It
has another irreducible factor of degree $6$ raised to the third power, which is the
Chow form of the singular locus of $D_A=0$ corresponding to degree $3$ polynomials with a triple root (a degenerate
multiple root), predicted by Theorem~\ref{th:main}.
}
\end{example}

\section{The Mixed Discriminant and the discriminant of the Cayley configuration}\label{section:mixed}

In this section we show in Proposition~\ref{prop} that the mixed discriminant  $MD(A_0,\ldots,A_r)$  
coincides in the non-defective case with the discriminant of the associated  Cayley configuration $D_C$, thus generalizing Theorem 2.1 in~\cite{CCDRS}.  
Note that when $C$ is defective these varieties need not coincide, as shown in Example~2.2 in~\cite{CCDRS}. We also characterize, 
in Proposition~\ref{prop:wz}, non-defectivity of $C$ when all $A_i$ are equal. The latter result relies on a classical criterion by Katz, 
stated as Lemma~\ref{K} below, and is a simple consequence of \cite[Th. 0.1]{WZ94}.

Recall that for a projective variety $X$, the \emph{dual defect} of $X$ is defined to be
\begin{equation}\label{eq:def}
\text{def}(X):=\text{codim}(X^\nu)-1,
\end{equation}
where $X^\nu$ is the dual variety consisting of singular hyperplane sections to $X$. 
In particular, if the dual variety is a hypersurface as expected, then the dual defect is equal to $0$ and $X$ is said to be
non-defective.  When $X=X_A$ for some finite lattice configuration $A$, we also say that $A$ is non-defective.
In this context, we have the following  lemma due  to Katz. 

\begin{lemma}\cite{Katz}\label{K}
Let $A\subset\Z^n$ be a lattice configuration with $|A|=N+1.$
Let $H_p(f)$ denote the Hessian matrix of an $A$-polynomial $f$. Then
$${\rm codim}{X_A^\nu}=1+{\rm min}_{f}\left({{\rm corank}(H_u(f))}\right)$$
where $u$ is a general point and $f$ varies among the polynomials  with support in  $A$  vanishing at $u.$
\end{lemma} 
In particular, ${\rm codim}{X_A^\nu}=1$ implies that polynomials vanishing at a general point $u$ together with their partial derivatives have Hessian of maximal rank.

Observe that Lemma~\ref{K} is equivalent to saying that in the non-defective case, the closure of the singular $A$-polynomials 
coincides with the closure of the nodal $A$-polynomials, that is, polynomials only admitting non-degenerate multiple roots.

\begin{corollary}\label{node}
If $A$ is a non-defective finite lattice configuration, then
\[X_A^\nu=\overline{\left\{f\in {\PP^N}^\nu\, : \, f(u)=0, \frac{\partial f}{\partial x_i}(u)=0\text{ and } \det (H(f))(u)\neq 0\text{ for some }u\in (K^*)^n\right\}}.\]
\end{corollary}
\begin{proof} The inclusion ``$\subseteq$'' follows by definition and the inclusion ``$\supseteq$'' follows by Lemma \ref{K}.    \end{proof}

Let us now consider the Cayley configuration $C$ associated to $r+1$ finite lattice
configurations $A_0,\ldots,A_r$ in $\Z^n.$  
We remark that the $r$ in \cite{GKZ} corresponds to our $r-1.$
We use the following notation: $(\lambda,x)=(\lambda_0, \ldots, \lambda_r, x_1, \dots, x_n)$.
A polynomial $f$ with support on $C$ has the form
\[  f \, = \,  \sum_0^r \lambda_i p_i,\]
where $p_i$ are $A_i$-polynomials in the variables $x$. Consider the Jacobian matrix  $\begin{bmatrix} \nabla_x(p_0)(u)&\ldots &\nabla_x(p_r)(u)\end{bmatrix}$
of $p_0, \dots, p_r$ at $u$.
Notice that $f \in X_C^\nu$ if there exists
\[(\lambda,u) \in (\C^\nu)^{r+1} \times
(\C^\nu)^n \text{ s.t. } p_0(u) = \dots = p_r(u)=0 \text{ and } \lambda\in \ker\begin{bmatrix} \nabla_x(p_0)(u)&\ldots &\nabla_x(p_r)(u)\end{bmatrix}^T;\]
 or equivalently, if $\sum_{i=0}^r \lambda_i \nabla_x(p_i)(u)=0$ and thus the gradients are linearly dependent.
  In particular, 
 $${\rm rank}( \begin{bmatrix} \nabla_x(p_0)(u)&\ldots &\nabla_x(p_r)(u)\end{bmatrix} )\le r.$$
We will now prove that the locus where the rank is exactly $r$ characterizes the dual variety $X_C^\nu,$ assuming it is a hypersurface.

\begin{proposition} \label{prop}  Let $A_0, \dots,A_r$ and $C$ as above and assume that $C$ is non-defective.
 Then, 
$$MD(A_0,\ldots, A_r) (p_0, \dots, p_r)=D_C (\sum_{i=0}^r \lambda_i p_i),$$
where $p_i$ are $A_i$-polynomials for $i=0,\dots, r$ and $(\lambda_0, \dots, \lambda_r)$ are variables. 
 \end{proposition}

\begin{proof}

Let $\phi_f$ be the tuple of coefficients of $f$. Corollary \ref{node} implies that 
$$X_C^\nu=\overline{\{ \phi_f \; :\;  f(\lambda,u)=0, p_{A_i}(u)=0,  i=0, \dots, r, \, 
\sum_{j=0}^r\lambda_j \frac{\partial p_{A_j}}
{\partial x_i}(u)=0 \text{ and } \det (H(f))(\lambda,u)\neq 0 \}},$$
where $(\lambda,u) \in (\C^\nu)^{r+1} \times(\C^\nu)^n.$ Here, $H(f)(\lambda,u)$ means the following: as $f$ is
homogeneous in the variables $\lambda$ and $\lambda \in (\C^\nu)^{r+1}$,  we assume that $\lambda_0=1$ and that
$(\lambda_1, \dots, \lambda_r)$ are its affine coordinates. Thus, $H(f)(\lambda,u)$ is the Hessian of $f$
with respect to the variables $(\lambda_1, \dots, \lambda_r, x_1,\dots, x_n)$ . 
This Hessian matrix is of the form
\[  H_{(u,\lambda)}(\phi_f)= \begin{bmatrix}    \sum \lambda_i H(p_{A_i})(u)&  
 \begin{bmatrix} \nabla(p_{A_1})(u)  \\ \ldots \\   \nabla(p_{A_r})(u)    \end{bmatrix}^T \\
\begin{bmatrix} \nabla(p_{A_1})(u)  \\ \ldots \\   \nabla(p_{A_r})(u)    \end{bmatrix} &0
\end{bmatrix}.  \]
It follows that if $\phi_f\in X_C^\nu$  and $\det H(f)(\lambda, u) \neq 0,$ which happens for generic points in $X_C^\nu$
 by Corollary~\ref{node}, then $\textrm{rank} \begin{bmatrix} \nabla(p_{A_1})(u)  \\ \ldots \\   \nabla(p_{A_r})(u)\end{bmatrix} =r$,
 that is, 
 the gradients of the polynomials $p_{A_i}(u)$ for $i\neq 0$ form a matrix of rank $r$, that is, they are linearly independent.
 This happens similarly for the gradients of any subset of $r$ polynomials $p_{A_j}(u)$.
 Moreover, 
this is exactly the condition implying that $\phi_f$ belongs to the mixed-discriminantal variety $MD(A_0,\ldots, A_r)=0$
which we denote by  $X_{MD}$. It follows that $X_C^\nu\subseteq X_{MD}$ and that 
$X_{MD}$ is also a hypersurface, i.e. $MD(A_0,\ldots, A_r)\neq 1.$

The reverse inclusion follows essentially from the definition. In fact if $\phi_f\in X_{MD}$ is generic,
then there is a  common zero $u \in (C^\nu)^n$ of $p_{A_0}, \dots, p_{A_r}$ 
and a linear dependency $\sum \lambda_i \nabla(p_{A_i})(u)=0$ with all $\lambda_i \neq 0,$
because all the maximal minors in the matrix
$\begin{bmatrix} \nabla(p_{A_0})(u)  \\ \ldots \\   \nabla(p_{A_r})(u)    \end{bmatrix}$
are assumed to be nonzero. It follows that  $\phi_f\in    X_C^\nu.$ \end{proof}

 Notice that if $A_0=A_1=\ldots=A_r=A$ then $C=\{e_0, \dots, e_r\} \times A$, which is usually written as $C= \Delta_r \times A$.  Following~\cite{WZ94}, we define the following quantity
 associated to a projective variety $X$:
 \begin{equation}\label{eq:mu}
 \mu(X) \, = \, \dim(X) + {\rm def}(X),
 \end{equation}
where the defect of $X$ has been defined in~\eqref{eq:def}. We end this section with the following
result about non-defectivity.

\begin{proposition}\label{prop:wz}
Let $A$ be a non-defective finite lattice configuration.
Then,  the associated Cayley configuration $C= \Delta_r \times A$ is non-defective if and only if
 $r \le \dim(X_A)$.
\end{proposition}
 
 \begin{proof}
 Note that $X_C$ equals the Segre embedding of $\PP^{r}\times X_A$. We can then use Theorem 0.1 in~\cite{WZ}, which says that
 $$ \mu(\PP^{r}\times X_A)= {\rm max}(r+\dim(X_A), r+ {\rm def}(\PP^{r}), \dim(X_A) + {\rm def}(X_A)).$$
 According to~\eqref{eq:mu}, we have that  $\mu(X_C) = r + \dim(X_A) + {\rm def}(X_C)$. Since 
 ${\rm def}(\PP^{r})=r$, and by hypothesis ${\rm def}(X_A)=0$, we get that
 $$\mu(X_C) = {\rm max}(r+\dim(X_A), 2r, \dim(X_A)). $$
 When  $r \le \dim(X_A)$, we get that $\mu(X_C) = r + \dim(X_A)$   
 which implies that ${\rm def}(X_C)=0$. On the other side, when $r > \dim(X_A)$,
 we have that  $\mu(X_C) = 2r$ and so ${\rm def}(X_C)= r - \dim(X_A) >0$.
 \end{proof}

\section{The multivariate iterated discriminant}\label{section:iterated}
In the remainder of the paper we will consider the case  $A_i=A$ for $i=0,\ldots,r.$ In order to establish an iterated process for the mixed 
discriminant it is convenient  to consider the {\em  geometric iterated discriminant}  $JD_{r,A}$ introduced in Definition~\ref{def:Sigma} below.
In  Proposition~\ref{discr}  we  prove  that this polynomial coincides with the iterated discriminant $ID_{r,A}$ from Definition~\ref{def:id}.  
It implies that Theorem~\ref{th:main}, which can be considered the main result of this paper, also holds for $ID_{r,A}$, as stated in the Introduction. 

Recall that given an irreducible and reduced projective variety $Y\subset\PP^N$ of codimension $s$, its Chow form $Ch_Y$ is defined as follows.
Consider linear subspaces of dimension $\ell$ in $\PP^N$,  $L\in Gr(\ell+1,N+1).$ If $s\geq \ell+1$, any generic $L$  will not intersect $Y.$ The
irreducible subvariety \[\{L\in Gr(\ell+1,N+1)\, :\, L\cap Y\neq\emptyset\}\] parametrizing the exceptional intersection locus,  has codimension $(s-\ell)$ in $Gr(\ell+1,N+1).$ 
In  case $\ell=s-1$ the defining polynomial is denoted by $Ch_{Y}$ and it is called the {\em Chow form} of $Y$ \cite[page~99]{GKZ}.

We also need to recall two classical facts which will be used in the proof of our main Theorem~\ref{th:main}.

\begin{remark}\label{rem}
{\rm Given a finite lattice configuration $A$ and a generic singular hyperplane section of $X_A$, we can recover the
intersection point by means of the gradient of the discriminant $D_A$. Precisely, 
\begin{enumerate}
\item As we are assuming that ${\rm char}(K)=0$, if a regular point $H$ in the dual variety $X_A^\nu$ is tangent to $X_A$ at a regular point $y_H$, 
then this projective point is unique and $y_H=\nabla D_A(H)$ \cite[Th.1.1, Ch. 1]{GKZ}.
This is referred to as {\em biduality}.
\item When  $X_A^\nu\subset(\PP^N)^\nu$ is a hypersurface, biduality implies that the  Gauss map $\gamma: X_A^\nu\dashrightarrow \PP^n$  
 is defined by $H\mapsto  \nabla D_A(H)=y_H$ and the closure of its image equals $X_A$.
\end{enumerate}
}
\end{remark}

Let $A=\{a_0,\ldots,a_N\}\subset\Z^n$ be a lattice configuration. We will assume henceforth that  $A$ is non-defective
and that $D_A$ is a homogeneous polynomial of degree $\delta >0.$  

Given $A$-polynomials 
$$p_i=\sum_{j=0}^N c_{ij}x^{a_j},\quad\quad i=0,\ldots r,$$ we also denote by $(p_0,\ldots,p_{r})\in(\PP^{(r+1)(N+1)-1})^\nu$ 
the vector of their coefficients.  For any $\lambda=(\lambda_0,\ldots,\lambda_r)\in\PP^{r}$ 
we write $$P_\lambda:=\lambda_0p_0+\ldots + \lambda_r p_r\in(\PP^N)^\nu.$$

\begin{definition}\label{def:Sigma}
Consider the incidence variety
\begin{equation}\label{Sigma}\Sigma=\left\{((p_0,p_1,\ldots,p_r),\lambda)\in (\PP^{(r+1)(N+1)-1})^\nu\times 
\PP^{r} \, : \, \sum_j c_{ij}\frac{\partial D_A}{\partial c_j}(P_\lambda)=0, \,  i=0,\ldots,r\right\}.\end{equation}
Let $\pi:\Sigma\to (\PP^{(r+1)(N+1)-1})^\nu$ be the linear projection onto the first factor.
 The {\em $r$-multivariate iterated dual scheme} $\pi(\Sigma)$  is defined by the projective elimination ideal
\begin{equation} \label{eq:pi} 
\pi I=(I:m^{\infty})\cap \mathbb{C}[c], \end{equation}
where $\mathbb{C}[c]$ is the ring of polynomials in the variables $c_{ij}$, $m$ is the the irrelevant ideal of $\PP^r$, and $I$ is the ideal  $$I=\left\langle\sum_j c_{ij}\frac{\partial D_A}{\partial c_j}(P_\lambda), i=0,\ldots,r \right\rangle.$$
When $\pi(\Sigma)$  has codimension one,  we denote by  $JD_{r,A}\in \Z[c]$ a generator (unique up to multiplication by a nonzero constant) of the union of the codimension one 
components of $\pi I$ and we call it the \emph{geometric iterated discriminant}.
\end{definition}
Notice that the projection is in general not irreducible;  see for instance Example \ref{ex:12}. We will see in
Proposition~\ref{discr} below that the geometric iterated discriminant $JD_{r,A}$ coincides with the
more naive definition of the iterated discriminant $ID_{r,A}$ from Definition~\ref{def:id}.

Let $(p,\lambda)=((p_0,\ldots,p_r),(\lambda_0,\ldots,\lambda_r))\in\Sigma$. 
In order to understand the projection $\pi$ we consider two auxiliary maps, $\phi: \Sigma\to X_A^\nu$ and $T: (\PP^{(r+1)(N+1)-1})^\nu\dashedrightarrow Gr(r+1, N+1):$
\[\begin{tikzcd}
 \Sigma \arrow{r}{\phi} \arrow[swap]{d} &  X_A^\nu  \\
\pi(\Sigma) \arrow[dashed]{r}{T} & Gr(r+1, N+1)
\end{tikzcd}
\]
defined by $\phi(p,\lambda)=P_\lambda$ and $T(p_0,\ldots, p_r)=T_p,$ where we denote by $T_p$ the projective linear span of $p_0,\ldots, p_r.$

\begin{lemma} \label{lem:euler}  Let $p=(p_0,\ldots,p_r)\in(\PP^{(r+1)(N+1)-1})^\nu$ such that  $JD_{r,A}(p)=0$. Then, $T_p$ is tangent to $X_A^\nu$ at some point $\xi.$
\end{lemma}
\begin{proof} If  $JD_{r,A}(p_0,\ldots,p_r)=0$ then there exists $\lambda$ such that $(p,\lambda)\in\Sigma$; let $\xi=P_\lambda.$  Consider the  equalities
\begin{equation}\label{Euler}
 0 =\sum_j c_{ij}\frac{\partial D_A}{\partial c_j}(P_\lambda), \quad i=0,\ldots,r .\end{equation}
 Note that  these equations equal the derivatives with respect to $\lambda_0, \dots, \lambda_r$ of the composed function $(f, \mu) \to D_A(\sum_{i=0}^r \mu_i f_i)$ at
 the point $(p, \lambda)$.
 The
Euler relation  implies that $D_A(P_\lambda)=0$,
and thus $P_\lambda\in X_A^\nu.$ Moreover~\eqref{Euler}   implies that each $p_i$ lies in $T_{X^\nu,\xi}$, which is equivalent to $T_p\subset T_{X^\nu,\xi}$.  \end{proof}

Recall that we denote by $\textrm{sing}(X_A^\nu)$ the subscheme of $X_A^\nu$ defined by the ideal generated by the partial derivatives of $D_A.$ 

\begin{theorem}\label{th:main} Assume $A \subset \Z^n$ is non-defective and let $r\in\Z$ with $0 \le r \le \dim(X_A).$ 
Then,  the mixed discriminant $MD_{r,A}$  always divides the geometric iterated discriminant $JD_{r,A}$. Moreover:
\begin{enumerate}
\item If ${\rm codim}_{X_A^\nu}(sing(X_A^\nu))> r$, then $JD_{r,A} = MD_{r,A}$.
\item  \label{it:2} If ${\rm codim}_{X_A^\nu}(sing(X_A^\nu))= r$, then $JD_{r,A} = MD_{r,A}  \prod_{i=k}^\ell Ch_{Y_k}^{\mu_k},$
where $Y_1, \dots, Y_\ell$ are the irreducible components of $sing(X_A^\nu)$ of maximal
dimension $r$, with respective multiplicities $\mu_k\ge 2$.
\item If ${\rm codim}_{X_A^\nu}(sing(X_A^\nu)) < r$,
then $\pi(\Sigma)=(\PP^{(r+1)(N+1)-1})^\nu,$ and $JD_{r,A}=0.$
\end{enumerate}
\end{theorem}

\begin{proof} As already observed, in the classical case of $r=0$ we have $ID_{0,A}=MD_{0,A}=D_A.$
Note also that  by Propositions~\ref{prop:wz} and~\ref{prop}, $\deg(MD_{r,A})>0$ and it is irreducible.
 
Observe that the map $\phi$ is surjective since for any $F\in X_A^\nu,$ 
$F=\phi(F,\ldots,F,\frac{1}{(r+1)},\ldots, \frac{1}{(r+1)})$ and  we have $(F,\ldots,F,\frac{1}{(r+1)},\ldots, \frac{1}{(r+1)})\in\Sigma.$  
The rational map $T$ is defined over the open dense subset $U_T=\{p\; : \, T_p \simeq \PP^{r}\}$ of all  $p$ with linear span of projective dimension $r$. 
  Notice also that $T$ is surjective and that for each  $H\in Gr(r+1, N+1)$, the fiber $T^{-1}(H)$ has dimension $(r+1)^2-1.$ Let  
$\Sigma^\circ=\phi^{-1}((X_A^\nu)_{reg})$ and $\Sigma'=\phi^{-1}(sing(X_A^\nu))$; that is, let
$$\Sigma^\circ=\{(p, \lambda)\in \Sigma \, :\, P_\lambda\in (X_A^\nu)_{reg}\}, \quad \Sigma'=\{(p, \lambda)\in \Sigma \, :\, P_\lambda\in sing(X_A^\nu)\}.$$
It follows that $\pi(\Sigma)=\pi(\Sigma^\circ)\cup\pi(\Sigma').$

We claim that $\overline{\pi(\Sigma^\circ)}\subseteq V(MD_{r,A}).$
In fact, take a generic point $(p,\lambda)\in\pi(\Sigma^\circ)$.
We can then assume that not only $P_\lambda\in(X_A^\nu)_{reg}$,
but also there is a unique
regular point $y=(x^{m_0}:\ldots:x^{m_N})\in X_A$ with $x \in (K^*)^n$ such that $P_\lambda(x)=0$ and $\frac{\partial P_\lambda}{\partial x_i}(x)=0, \,\, i=1,\ldots, n.$ 
By Remark \ref{rem},  $y=(\frac{\partial D_A}{\partial c_0}(P_\lambda):\ldots: \frac{\partial D_A}{\partial c_n}(P_\lambda)).$ 
The equations $\frac{\partial P_\lambda}{\partial x_i}(x)=0$ for $i=1,\ldots,n$  mean that   $\sum \lambda_i \nabla(p_i)(x)=0.$ 
Moreover  $p_i(x)=0$ for all $i$ because
\begin{equation}\label{regular}p_i(x)= \sum_j c_{ij} y_j=k\sum_j c_{ij}\frac{\partial D_A}{\partial c_j}(P_\lambda)=0, \quad i=0,\ldots r \end{equation}         
 for some $k\in K^*$ such that  $y=k \, \nabla D_A(P_\lambda).$
This implies that $MD_{r,A}(p_0,\ldots,p_r)=0.$ 

We now show that $V(MD_{r,A})\subseteq \pi(\Sigma).$
Let $(p_0,\ldots,p_r)$ be a generic element in the zero locus of  $MD_{r,A}.$ Then there exists $(u,\lambda)\in (K^*)^{n+r+1}$ such that
\[ p_0(u)=\cdots=p_r(u)=0\text{ and } \begin{bmatrix} \nabla p_0(u)&\cdots&  \nabla p_r(u)\end{bmatrix}\begin{bmatrix}\lambda_0\\ \vdots\\\lambda_r 
 \end{bmatrix}=0.\] We claim that $(p,\lambda)\in\Sigma.$
If  $P_\lambda\in sing(X_A^\nu)$ then $\frac{\partial D_A}{\partial c_j}(P_\lambda)=0$ and thus $(p_0,\ldots,p_r)\in\pi(\Sigma)$. If instead $P_\lambda\in (X_A^\nu)_{reg}$ is generic,
then biduality gives $\nabla D_A(P_\lambda)=y$ with $y=(u^{m_0}:\ldots:u^{m_N})$  and thus $\sum_j c_{ij}\frac{\partial D_A}{\partial c_j}(P_\lambda)=p_i(u)=0$ 
as in \eqref{regular},  implying again that $(p_0,\ldots,p_r)\in\pi(\Sigma).$ We have then proved that
\begin{equation}\label{contained} \overline{\pi(\Sigma^\circ)}\subseteq V(MD_{r,A})\subseteq \pi(\Sigma).\end{equation}

Consider the non-embedded primary components of the ideal $\langle \frac{\partial D_A}{\partial c_j}, j=0,\ldots,N \rangle$ defining the singular locus of $X_A^\nu.$ 
Correspondingly, we consider  the decomposition into irreducible components  $sing(X_A^\nu)= \bigcup Y_k.$  Define
\begin{equation}\label{eq:Vk}
V_k=\{H\in Gr(r+1,N+1)\,: H\cap Y_k\neq \emptyset\}\text{ and } \Sigma_k=\phi^{-1}(Y_k).\end{equation}
Recall that $\codim_{Gr(r+1,N+1)}(V_k)=\max\{0, \codim_{\PP^{N}}(Y_k)-r\}.$ 

 Assume that $\codim_{X_A^\nu}(sing(X_A^\nu))> r$.  Then $\codim_{Gr(r+1,N+1)}(V_k)\geq 2$ for all $k.$ It follows that 
 for all $i$
 $$\codim_{{\PP^{(r+1)(N+1)-1}}^\nu}
 (\pi(\Sigma_i))=\codim_{{\PP^{(r+1)(N+1)-1}}^\nu}(\overline {\pi(\Sigma_i)\cap U_T})\geq \codim_{{\PP^{(r+1)(N+1)-1}}^\nu}(\overline{T^{-1}(V_i)})\geq 2.$$ 
 The containment in Equation  \eqref{contained} then implies that
$\pi(\Sigma)$ is  of codimension one and set-theoretically coincides with $V(MD_{r,A})$.
If $\codim_{X_A^\nu}(sing(X_A^\nu))=r$, then $\codim_{Gr(r+1,N+1)}(V_k)=1$ and thus by definition 
$\pi(\Sigma_k)=V(Ch_{Y_k}^{\mu_k})$ for some integer exponents $\mu_k$. We prove that these multiplicities are at least equal to $2$ in Proposition~\ref{prop:muk} below.

As the mixed discriminant $MD_{r,A}$ is irreducible, it remains to show that the multiplicity of  $MD_{r,A}$ in $ID_{r,A}$ is equal to $1$.  For that, it is enough to show that 
there exists $(p_0^\nu,\ldots,p_r^\nu)\in V(MD_{r,A})$ and $\lambda^\nu$ such that $(p^\nu,\lambda^\nu)\in\Sigma$ and $d\pi ((p^\nu,\lambda^\nu))$ has maximal rank.

We start by choosing a  point  $\xi\in reg(X_A^\nu)$ such that $\textrm{rank}(H)=n+1,$ where  $H=Hess(D_A)(\xi).$ Notice that $H={\rm Jac}(\gamma)(\xi),$ where 
$\gamma:X_A^\nu \dasharrow X_A$ is the Gauss map defined as $\gamma(y)=\nabla D_A(y)$ which in affine coordinates has generic rank equal to $n=\dim(X_A)$.  Up to a change of coordinates, 
$H$ can be assumed to be of the form \begin{equation}\label{hess} H= \begin{bmatrix} I_{n+1}&0\\0&0  \end{bmatrix}. \end{equation}

Consider $Z=\{M\in Gr(r+1, N+1) \,:\, \xi\in M\subset T_{X_A^\nu,\xi} \}$.
Note that for every   $p\in T^{-1}(Z),$  we have $\xi\in T_p\subset T_{X_A^\nu,\xi}$ and $\dim(T_p)=r$  as $p\in U_T$.
 It follows that there exists a unique $\lambda^*$ such that $(p,\lambda^*)\in\Sigma$ and $\phi(p,\lambda^*)=\xi.$ 
Consider the $(r+1)\times (N+1)$ matrix $C_p$ whose $i$-th row corresponds to the coefficients of $p_i$.
 Without loss of generality, the matrix $C_p$ can be assumed to be of the form  $[I_{r+1}, C].$ 

The matrix $M$ of the lifted linear map $\pi:  (\C^{(r+1)(N+1)})^\nu\times \C^{r+1} \to (\C^{(r+1)(N+1)})^\nu$ is equal to
$ M= [I_{(r+1)(N+1)}, 0]$, where $0$ denotes the zero matrix of size $ (r+1)(N+1) \times (r+1)$.  Let us call
 $g_0, \dots g_r$ the defining equations of $\Sigma$ in~\eqref{Sigma}.  It follows that $d\pi (p,\lambda^*)$ is of maximal rank if 
 and only if the square $(r+1)(N+1)+ (r+1)$-matrix $M'$ with upper rows consisting of the Jacobian of $g_0, \dots, g_n$ with respect to the
  variables $(c_{01}, \dots, c_{0N},  \dots, c_{r0}, \dots, c_{rN},  \lambda_0, \dots, \lambda_r)$  evaluated at $(p, \lambda^*)$ 
 and lower rows given by the matrix $M$, has maximal rank. 
 But given the form of $M$, this is equivalent to the fact that the $(r+1)\times(r+1)$ submatrix $H^*$  at the right upper corner of $M'$ has maximal rank $r+1$. 
 Recall from the proof of Lemma~\ref{lem:euler} that $g_0, \dots g_r$  equal  the derivatives with respect to $\lambda_0, \dots, \lambda_r$ 
 of the composed function $(f, \lambda) \to D_A(\sum_{i=0}^r \lambda_i f_i)$. Then, $H^*$ consists of the 
 Hessian matrix with respect to the $\lambda$-variables of this composed function. Therefore, we have that
 \begin{equation}\label{eq:H}
 H^* \, = C_p \, H \, C_p^t.
 \end{equation}
Recall that we assume that 
  $r\leq n$, and thus $r+1 \le n+1$. Given the form of the coefficient matrix $C_p$ and of the Hessian matrix $H$ in~\eqref{hess}, 
  we deduce that $H^*$ is of maximal rank because it is the identity matrix $I_{r+1}$.

Assume that $\codim_{X_A^\nu}(sing(X_A^\nu))< r$. Then $\codim_{\PP^{N}}(Y_k)<r+1$ for all $k.$ The assumption 
also implies that any element of the Grassmannian belongs to $V_k$ for all $k$ (defined in~\eqref{eq:Vk}) and that $\pi(\Sigma')\cap U_T= T^{-1}(Gr(r+1,N+1))=U_T.$ It follows that $\PP^{(r+1)(N+1)-1}
=\overline{U_T}=\overline {\pi(\Sigma_i)\cap U_T}=\pi(\Sigma_i)\subset \pi(\Sigma)$ and thus $\pi(\Sigma)=(\PP^{(r+1)(N+1)-1})^\nu$.
 \end{proof}

The following Proposition~\ref{discr} explains the name geometric iterated discriminant: we show that under the hypotheses  
of Theorem~\ref{th:main}, the polynomial $JD_{r,A}$ in Definition~\ref{def:Sigma}  equals the
polynomial $ID_{r,A}$ in Definition~\ref{def:id}, and thus when it is nonzero it can be 
computed as a discriminant of a discriminant.  

Recall that, given a natural number $d$, we denote by  $d\Delta_{r}$ in $\R^{r+1}$  the lattice configuration given by the integer points in the dilated unit simplex $d$ times, and by 
$D_{d\Delta_{r}}$ the associated discriminant. For any homogeneous polynomial $H= H(\lambda_0, \dots, \lambda_r)$ of degree $d$, 
the discriminant of $H$ equals, up to constant, the resultant of its partial derivatives:
\begin{equation}\label{eq:D=Res}
D_{d\Delta_{r}}(H)  \, = \, {\rm Res}_{d-1}\left(\frac{\partial H}{\partial \lambda_0},
\dots, \frac{\partial H}{\partial \lambda_r}\right),
\end{equation}
where ${\rm Res}_{d-1}$ denotes the homogeneous resultant associated to $r+1$ homogeneous polynomials of degree $d-1$ (see Prop.~1.7, Ch.~13 in~\cite{GKZ}). 
Moreover, the following universal property is proved  in~\cite{Jou} (see Theorem~3.8 in \cite{Buse} for an English concise version). Let $G_0, \dots, G_r \in \Z[u][\lambda_0, \dots, \lambda_r]$ have degree $d-1$ with generic coefficients $u$:
\[ G_i (u, \lambda) = \sum_{|\alpha|=d_i} \, u_{i, \alpha} \, \lambda^\alpha.\]
Denote by $I_G$ the ideal $\langle G_0, \dots, G_r \rangle \subset \Z[u][\lambda_0, \dots, \lambda_r]$ generated by $G_0, \dots, G_r$. Then,
${\rm Res}(G_0, \dots, G_r)$ is a generator of the (generic) projective elimination ideal
\begin{equation} \label{eq:piG} 
\pi I_G = (I_G: m^\infty) \cap \Z[u]. \end{equation} 
In particular. 
for any variable $\lambda_i$ and any $N > \sum_i d_i -n$, it holds that
\begin{equation}\label{eq:Resideal}
\lambda_i^N  \, {\rm Res}_{d-1}(G_0, \dots, G_r) \in \langle G_0, \dots, G_r \rangle.\end{equation}
  Thus, such an equality holds for any specialization of the coefficients $u$ in a ring.

Let $A$ be a non-defective configuration with
$\codim(sing(X_A^\nu))\geq r$.  Call $\delta= \deg(D_A)$ and for a choice of $A$-polynomials $p_0, \dots, p_r$ 
consider the evaluation $D_A(P_\lambda)=D_A(\sum_{i=0}^r \lambda_i p_i)$, which is either zero or a homogeneous
polynomial in $\lambda=(\lambda_0, \dots, \lambda_r)$ of degree $\delta$.

\begin{proposition}\label{discr}
Under the hypotheses of Theorem~\ref{th:main}, 
the following equality holds:
$$ JD_{r,A} = ID_{r,A}.$$
Moreover, when ${\rm codim}_{X_A^\nu}(sing(X_A^\nu)) \ge r$, the degree of the iterated discriminant 
equals
\begin{equation}
\label{deg}\deg(ID_{r,A}) = (r+1)\delta (\delta -1)^{r}.
\end{equation}
\end{proposition}

\begin{proof}
By Theorem~\ref{th:main} and Definition~\ref{def:id}, we can assume that $ {\rm codim}_{X_A^\nu}(sing(X_A^\nu)) \ge r$.

Let $(p_0^0, \dots, p_r^0,\lambda^0)$ be a point in the incidence variety $\Sigma$ defined in~\eqref{Sigma}.
Note that  for any $i=0, \dots,r$,  we have that
\[0 \, = \, \sum_j c_{ij}\frac{\partial D_A}{\partial c_j}(\sum_{i=0}^r \lambda_i^0 p_i^0) \, 
= \frac{\partial}{\partial \lambda_i} D_A(\sum_{i=0}^r \lambda_i p_i^0)(\lambda^0). \] 
Then, $ID_{r,A} =D_{\delta\Delta_{r}}(D_A(\sum_{i=0}^r \lambda_i p_i^0))=0.$ As we pointed out in~\eqref{eq:D=Res}, this homogeneous discriminant equals, up to constant, the resultant
$${\rm Res}_{\delta -1}\left(\frac{\partial}{\partial \lambda_0} D_A(\sum_{i=0}^r \lambda_i p_i^0), \dots, \frac{\partial}{\partial \lambda_r} D_A(\sum_{i=0}^r \lambda_i p_i^0)\right).$$
It then follows that if $D_{\delta\Delta_{r}}(D_A(\sum_{i=0}^r \lambda_i p_i^0))=0$, then there exists $\lambda^0 \in \PP^r$ which is a common zero of all these partial derivatives.
We deduce from Equation~\eqref{eq:Resideal} that   for any ring $R$ containing the coefficients of $D_A(P_\lambda)$ and for any $i=0, \dots, r$, 
the iterated discriminant $D_{\delta \Delta_r}(P_\lambda)$ lies in the ideal generated by the partial derivatives
$\frac{\partial}{\partial \lambda_j} D_A(\sum_{i=0}^r \lambda_i p_i^0),  j=0, \dots, r, $ in each localization $R[\lambda_0, \dots, \lambda_r]_{\lambda_i}$. Moreover, we have that the ideal $\pi I$ in~\eqref{eq:pi} is the specialization of the ideal
$\pi I_G$ in~\eqref{eq:piG}. Then, $ID_{r,A} = JD_{r,A}$, as claimed.

To see that Equation~\eqref{deg} holds, recall that  $\deg(D_A)=\delta$ and so the degree of 
$D_A( \sum_{i=0}^r \lambda_i p_i)$ in the coefficients of $p_0, \dots, p_r$ as well as in the $\lambda$ variables  is equal to $\delta$.  On the other side,
the degree of $D_{\delta\Delta_{r}}$ is equal to
$(r+1) \delta^r$. 
\end{proof}

\subsection{The exponents in Theorem~\ref{th:main}} \label{ssec:exponents}

We first present two examples that illustrate Theorem~\ref{th:main} with $r+1=2.$ In the first one, the singular locus has codimension $r+1=2$, which implies a factor 
(with multiplicity $2$) of the iterated discriminant. In the second one, the singular locus has codimension bigger than $2$, which implies equality  between $MD_{1,A}$ and $ID_{2,A}.$

\begin{example} \label{ex:12c}[Example~\ref{ex:12}, continued.]
{\rm Consider again the two dimensional configuration  corresponding to the first Hirzebruch surface ${\mathbb F}_1$:
$$A = \{ (0,0), (1,0), (2,0), (0,1), (1,1)\}.$$
Given a generic $A$-polynomial 
$f(x, y) = a_0 + a_1 x + a_2 x^2 + y (b_0 + b_1 x)$, we saw that 
$D_A(f) \, = a_0 b_1 ^2 - a_1 b_0 b_1 + a_2 b_0^2$.
The ideal defining the singular locus $S$ of $X_A^\nu$ is generated by
\[b_0 ^2, b_0 b_1, b_1^2, -a_1 b_1 + 2 a_2 b_0, 2 a_0 b_1-a_1, b_0.\]
This ideal has multiplicity $2$ and its radical is generated by $b_0, b_1$.
In this case, 
$ID_{2,A}$ has another irreducible factor $Ch_S$ of degree $2$ coming from the Chow form of $S$, to the second power:
\[ ID_{2,A} = MD_{1,A} \, \cdot \, Ch_S^2,\]
where $Ch_S ((a_0, a_1,a_2, b_0, b_1), (A_0, A_1, A_2, B_0, B_1)) = B_0 b_1- B_1 b_0$.}
\end{example}

\begin{example}\label{ex:1111}
{\rm
Let $X_A$  be the Segre embedding of $\PP^1 \times \PP^1 \times\PP^1$, so $D_A$ is the hyperdeterminant of format 
$(2,2,2)$ of degree $4$,  whose singular locus has codimension greater than $2$ by~\cite{WZ}. 
Take $r=2$, so that  $MD_{1,A}$ equals the discriminant of the hyperdeterminant of format $(2,2,2,2)$ 
(corresponding to the Segre embedding of $\PP^1 \times \PP^1 \times \PP^1 \times \PP^1$).
In this case, $MD_{1,A}$ equals the iterated discriminant $ID_{2,A}$ and thus has degree
$2\cdot 4\cdot (4-1)^1 =24$. 

This is the only known case of polynomials of degree bigger than $2$  for which the iterated and
the mixed discriminants coincide.}
\end{example}

We have not completely identified  the exponents $\mu_k$ occurring in the factorization of the iterated discriminant in Theorem~\ref{th:main} by 
the difficulties expressed in the Introduction, where we gave the only references to the literature we are aware of, but the following proposition shows that these exponents are strictly bigger than $1$.

\begin{proposition} \label{prop:muk} With notation and assumptions as in Theorem~\ref{th:main}, the exponents $\mu_k$ in item~\eqref{it:2} satisfy $\mu_k\geq 2$ for all $k$. 
\end{proposition}

\begin{proof} To simplify the notation, assume that ${\rm codim}_{X_A^\nu}(sing(X_A^\nu))= r=1$, $sing(X_A^\nu)=Y$ and  $JD_{r,A} = MD_{r,A}  Ch_{Y}^{\mu}$. 
Considering a generic point $(p_0,p_1)$ such that $MD(p_0,p_1)\neq 0$ and $Ch_{Y}(p_0,p_1)=0$ we have that:
$$\frac{\partial ID}{\partial c_{i,a}}(p_0,p_1)= \mu Ch_{Y}^{\mu-1} (p_0,p_1) MD_{1,A}(p_0,p_1)
 \frac{\partial Ch_{Y}}{\partial c_{i,a}}(p_0,p_1) + Ch_{Y}^{\mu}(p_0,p_1)\frac{\partial MD}{\partial c_{i,a}}(p_0,p_1)$$ for all $a\in A$ and $i=0,1.$
By genericity we may assume that $\frac{\partial Ch_{Y}}{\partial c_{i,a}}(p_0,p_1)\neq 0$ for some $c_{i,a}.$
We can conclude that if $\mu=1$ then $\frac{\partial ID}{\partial c_{i,a}}(p_0,p_1)\neq 0$ for some $c_{i,a}$ 
and thus that $\frac{\partial ID}{\partial c_{i,a}}(p_0,p_1)=0$ for all $c_{i,a}$ would imply $\mu\geq 2$.
We will prove that this is the case.

Recall that $Ch_{Y}(p_0,p_1)=0$ implies that the line spanned by $p_0,p_1$ intersects $Y$ at some point  which we denote by $p_0+\lambda^*p_1.$ 
This means that $D_A(p_0+\lambda^*p_1)=\frac{\partial D_A}{\partial c_{a}}(p_0+\lambda^*p_1)=0$ for all $a\in A.$
Let $D_A(p_0+\lambda p_1) =\sum_0^\delta \Gamma_j(c_{0,a}, c_{1,a})_{a\in A} \lambda^j$  and recall that $ID_{r,A}(p_0, \dots, p_r):=D_{\delta}(D_A(p_0+\lambda p_1))$ by Theorem \ref{discr}. It follows that:
$$\frac{\partial ID}{\partial c_{i,a}}(p_0,p_1)=\sum_1^\delta (\frac{\partial D_\delta}{\partial \Gamma_j}(D_A(p_0+\lambda^*p_1))\frac{\partial \Gamma_j}{\partial c_{i,a}}(p_0,p_1).$$

Observe that:
$$\frac{\partial D_A(p_0+\lambda^*p_1)}{\partial c_{0,a}}=\frac{\partial D_A}{\partial c_{a}}(p_0+\lambda^*p_1)=0, 
\frac{\partial D_A(p_0+\lambda^*p_1)}{\partial c_{1,a}}=\lambda^*\frac{\partial D_A}{\partial c_{a}}(p_0+\lambda^*p_1)=0.$$
Moreover:
$$0=\frac{\partial D_A(p_0+\lambda^*p_1)}{\partial c_{i,a}}=\sum_{j=0}^\delta  \frac{\partial \Gamma_j}{\partial c_{i,a}}(p_0,p_1)(\lambda^*)^j.$$

Recall that if $\frac{\partial D_\delta}{\partial \Gamma_j}(D_A(p_0+\lambda^*p_1)\neq 0$ for some $j$ then by biduality 
$\frac{\partial D_\delta}{\partial \Gamma_j}(D_A(p_0+\lambda^*p_1))=(\lambda^*)^j,$ which would conclude the proof.
If otherwise $\frac{\partial D_\delta}{\partial \Gamma_j}(D_A(p_0+\lambda^*p_1))=0$ for all $j$ then the assertion is also true.
  \end{proof}


Iterated discriminants with respect to one variable appear frequently in the study and applications of the Cylindrical Algebraic Decomposition proposed
 by Collins in 1975, and this lead to try to describe the singularities of discriminant hypersurfaces. In particular, 
the best detailed study is done by Bus\'e and Mourrain in Theorem~6.8 and Corollary~6.9 in~\cite{BM} for homogeneous polynomials of three variables 
(or more, but iterating twice the computation of a discriminant with respect to one of the variables) using resultants, with proofs that cannot be extended for general configurations $A$.

An interesting subsequent work is the paper by Lazard and McCallum~\cite{Lazard}. Again, they consider polynomials $f$ in variables $(x, y, z_1, \dots, z_m)$ 
and univariate iterated discriminants in $x$ and $y$ (thinking of $f$ in the ring $k[z_1,\dots, z_m][x,y]$) with rather elementary techniques. 
They identify the factors of their iterated discriminants but don't identify the exponents in general. However, they prove a series of very nice general and useful results,
 in particular Proposition~9 about the regular points of the discriminant (which is a version of biduality), and Propositions~10 through~14 about the singular points, 
 that could be used to identify the exponents $\mu_k$ in particular cases. 

 Based on the computation of different examples (see for instance Examples~\ref{ex:12} and
Example~\ref{ex:12c}), and the results in~\cite{BM} and~\cite{Lazard} that we mentioned, we see some evidence of the following. 

\begin{conjecture} \label{conj:muk}
The multiplicity $\mu_k$ are equal to $2$
 if $Y_k$ is a component of codimension one corresponding to the closure of the locus of  those $p$ for which there are  two different
non-degenerate multiple roots (the double point locus), while $\mu_k$ equals $3$ when $Y_k$ is a component of codimension one corresponding to the 
locus of those $p$  for which there is a degenerate multiple root (the cusp locus). 
\end{conjecture}

\section{Comparing mixed and iterated discriminants}\label{section:comparing}

In this section we consider the case when $A$ equals the lattice points in a cartesian product of dilates of standard simplices: $d_1 \Delta_{k_1}
\times\cdots\times d_\ell \Delta_{k_l}$, for some $\ell \ge 1$. In other words we investigate Segre-Veronese varieties $X_A= \mathbb{P}^{k_1}(d_1) \times\cdots\times \mathbb{P}^{k_\ell}(d_\ell).$

The symbol $\mathbb{P}^{k}(d)$ denotes  the Veronese embedding of degree $d$ in dimension $k$, i.e. the variety $\mathbb{P}^k$ 
embedded in $\mathbb{P}^{\binom{k+d}{d}-1}$ by the global sections of the line bundle ${\mathcal O}_{{\mathbb P}^k}(d).$ We occasionally denote $\mathbb{P}^{k}(1)$ by $\mathbb{P}^{k}.$

The symbol $\mathbb{P}^{k_1}(d_1) \times\cdots\times \mathbb{P}^{k_\ell}(d_\ell)$ denotes the Segre embedding of the above defined 
Veronese embeddings, more precisely the variety  $\mathbb{P}^{k_1} \times\cdots\times \mathbb{P}^{k_\ell}$ embedded  via the global
 sections of the line bundle $\pi^*_1 {\mathcal O}_{{\mathbb P}^{k_1}}(d_1)\otimes\ldots\otimes\pi_\ell^*{\mathcal O}_{{\mathbb P}^{k_{\ell}}}(d_\ell),$ where $\pi_i$ 
 denotes the $i^{th}$ projection $\pi_i:\mathbb{P}^{k_1} \times\cdots\times \mathbb{P}^{k_\ell}\to \mathbb{P}^{k_i}. $ These  are toric embeddings 
 corresponding to the configurations of lattice points of the polytopes $d_1\Delta_{k_1}\times\ldots\times d_l\Delta_{k_\ell}.$

When  $d_i=1$ we recover the case
of hyperdeterminants, which has been completely solved in~\cite{WZ}. In Proposition~\ref{proposition:r1d2} we show
that when $\ell=1$ there is equality if and only if $r=1$ and $d_1=2$. We then conjecture that these are all the possible
cases (see Conjecture~\ref{conjecture:main}), that is, in all other cases the singularities of the discriminantal locus have codimension
one in the dual variety. We conclude with Theorem~\ref{theorem:d_i>1}, which covers the case in which all $d_i > 1$.

To determine when the iterated and mixed discriminants of Segre-Veronese varieties are equal, we start with the following lemma,
which allows us to compute the degree of $MD_{r,d\Delta_n}$, that is, the case in which we consider $(r+1)$ polynomials of degree $d$ in $n$ variables. 
Recall that when $r \le n$, we know by Proposition~\ref{prop} that the mixed discriminant equals the discriminant of the Cayley configuration
given by the lattice points in the product of simplices $\Delta_r\times d\Delta_n$.

\begin{lemma}\label{lemma:mddeg} 
If $r \le n$ then $\deg(MD_{r,d\Delta_n})=(n+1)\binom{n}{r} d^r(d-1)^{n-r}$.
\end{lemma}

\begin{proof}  We will use  \cite[Ch. 13, Theorem 2.4]{GKZ}, 
which tells us that this degree is equal to the coefficient of the monomial $x^ry^n$ in the expansion of 
$$S(x,y)=\frac{1}{\left((1+x)(1+y)-x(1+y)-d y(1+x)\right)^2}=\frac{1}{(1-(d-1)y-dxy)^2}. $$
We may write
$$S(x,y)=\left(\frac 1{1-q}\right)^2= \sum_{n \ge 0} (n+1) q^n, $$
where   $q = (d-1)y+dxy = y ((d-1)+d x)$.

Since
\[ q^n = \left(\sum_{j=0}^n \binom{n}{j} (d-1)^{n-j} d^j x^j\right)y^n, \]
we have
\begin{equation}\label{s} S(x,y) \, = \, \sum_{n\ge 0}
 \sum_{j=0}^n (n+1)\binom{n}{j} (d-1)^{n-j} d^j x^j y^n.\end{equation}
From this expansion, we see that the coefficient of $x^r y^n$ is equal to
\[(n+1) \binom{n}{r} (d-1)^{n-r} d^r\]
when $r\leq n$, and  is equal to $0$ if $r > n$.  This completes the proof.

\end{proof}

\begin{proposition}\label{proposition:r1d2}  Let $1<d$ and $1\leq r\leq n$.  
Then $MD_{r,d\Delta_n}=ID_{r,d\Delta_n}$ if and only if $r=1$ and $d=2$.
\end{proposition}

The fact that this equality holds in the case of $r=1$ and $d=2$ was shown in \cite[Theorem 8.2]{O}.  Although we include this in our proof for completeness, 
the main contribution of this result is that equality does not hold in any other case.

\begin{proof}
For any $d >1$, the configuration of lattice points in $d \Delta_n$ is non-defective~\cite{BJ} and as $r \le n$, it is enough to check that
$\deg(MD_{r,d\Delta_n})=\deg(ID_{r,d\Delta_n})$ by Propositions~\ref{prop} and~\ref{prop:wz}.

From Lemma \ref{lemma:mddeg} we know that 
$$\deg(MD_{r,d\Delta_n)})=(n+1)\binom{n}{r} d^r(d-1)^{n-r}.$$
By Proposition \ref{discr} we also know that
$$\deg(ID_{r,d\Delta_n})=(n+1) (d-1)^n(r+1)((n+1)(d-1)^n-1)^r.$$
To determine when these are equal, we will consider 
the ratio of the two degrees, both of which are nonzero for $d>1$.  We have
$$\frac{\deg(ID_{r,d\Delta_n})}{\deg(MD_{r,d\Delta_n})}=\frac{(n+1) (d-1)^n(r+1)((n+1)(d-1)^n-1)^r}{(n+1)\binom{n}{r} d^r(d-1)^{n-r}}
=\frac{(d-1)^r(r+1)((n+1)(d-1)^n-1)^r)}{\binom{n}{r} d^r}.$$
For any $d>1$ we have
$$(n+1)(d-1)^n-1=n(d-1)^n+(d-1)^n-1\geq n(d-1)^n,$$
with equality if and only if $d=2$.  Thus the numerator satisfies
$$(d-1)^r(r+1)((n+1)(d-1)^n-1)^r\geq (d-1)^r(r+1)(n(d-1)^n)^r=(d-1)^{r(n+1)}(r+1)n^r.$$
Since $\binom{n}{r}\leq \frac{n^r}{r!}$, we have
$$\frac{\deg(ID_{r,d\Delta_n})}{\deg(MD_{r,d\Delta_n})}\geq \frac{(d-1)^{r(n+1)}(r+1)n^r}{\frac{n^r}{r!}d^r}=
\frac{(d-1)^{r(n+1)}(r+1)!}{d^r}=\left(\frac{(d-1)^{n+1}}{d}\right)^r \cdot (r+1)!$$
If $d=2$, then this ratio is $\frac{(r+1)!}{d^r}=\frac{(r+1)!}{2^r}\geq 1$, with equality if and only if $r=1$.  If $d>2$, 
then $(d-1)^{n+1}\geq(d-1)^2\geq 2(d-1)=2d-2>d$.  Thus $\frac{(d-1)^{n+1}}{d}>1$, and so
$$\frac{\deg(ID_{r,d\Delta_n})}{\deg(MD_{r,d\Delta_n})}> (r+1)!.$$
Thus except possibly in the case of $d=2$ and $r=1$, we have $\deg(ID_{r,d\Delta_n})>\deg(MD_{r,d\Delta_n})$.

To see that $d=2$ and $r=1$ gives $\deg(MD_{r,d\Delta_n)})=\deg(ID_{r,d\Delta_n})$, note that in this case the ratio of the degrees is
$$\frac{(2-1)^1(1+1)((n+1)(2-1)^n-1)^1)}{\binom{n}{1} 2^1}=\frac{2(n+1-1)}{2n}=1.$$

\end{proof}

Geometrically, Proposition \ref{proposition:r1d2} shows that for $\mathbb{P}^r(1)\times\mathbb{P}^n(d)$ the associated mixed discriminant is equal to the iterated discriminant only when
$r=1$ and $d=2$. Note that we don't consider the case $d=1$ because this case is defective.

It is natural to consider the same question for any product-of-simplices :
$$\mathbb{P}^{r}(1)\times\mathbb{P}^{k_1}(d_1)\times\cdots\times \mathbb{P}^{k_\ell}(d_\ell).$$
Setting $d=(d_1,\ldots,d_\ell)$ and $k=(k_1,\ldots,k_\ell)$, let $A_{\ell,d,k}$ denote the  configuration
corresponding to  $\mathbb{P}^{k_1}(d_1)\times\cdots\times \mathbb{P}^{k_\ell}(d_\ell)$. 

We conjecture the following:

 \begin{conjecture}\label{conjecture:main}
   We have $ID_{r,A_{\ell,d,k}}= MD_{r,A_{\ell,d,k}}$ if and only if $\mathbb{P}^{r}(1)\times\mathbb{P}^{k_1}(d_1)
\times\cdots\times \mathbb{P}^{k_\ell}(d_\ell)$ is of one of the following forms:
 \begin{enumerate}
 \item $\PP^r\times\PP^m\times\PP^m, m\geq 1, r=1,2$,
 \item $(\PP^1)^4$,
 \item $\PP^1\times\PP^n(2)$.
 \end{enumerate}
 \end{conjecture}
 
This conjecture was inspired by the question posed in   \cite[Chapter 14, pg 479]{GKZ}, which coincides with the above conjecture when $d_i=1$ for all $i$.  Their conjecture (and thus our 
 conjecture in this special case) was proved in \cite{WZ}.  Note that Proposition \ref{proposition:r1d2} implies that Conjecture \ref{conjecture:main} is true when $\ell=1$, which puts us into case (3).

 To study our conjecture in general, the following theorem giving the degree of the mixed discriminant 
$MD_{r,A_{\ell,d,k}}$ will be useful.  Let $B=B_{\ell}$ be the set of all non-empty subsets $\Omega\subset\{0,1,\ldots,\ell\}$.  For each $\Omega\in B$, let
 $$d_\Omega=\sum_{j\in\Omega}d_j.$$

  Let $\delta(\Omega)\in\mathbb{Z}^{\ell+1}_+$ be the characteristic vector of $\Omega$.  
For every $\kappa=(r,k_1,\ldots,k_\ell)\in\mathbb{Z}_+^{r+1}$, let $\mathcal{P}(\kappa)$ 
denote the set of all partitions of $\kappa$ into a sum of vectors $\delta(\Omega)$; in other words, $\mathcal{P}(\kappa)$
 is the set of all non-negative integral vectors $(m_\Omega)_{\Omega\in B}$ such that $\sum_{\Omega\in B}m_\Omega \delta(\Omega)=\kappa$.
 
 \begin{theorem}[Theorem 13.2.5, \cite{GKZ}]  The degree of $MD_{r,A_{\ell,d,k}}$ is given by
 $$\sum_{(m_\Omega)\in\mathcal{P}(\kappa)}\left(1+\sum_{\Omega\in B}
m_\Omega\right)!\prod_{\Omega\in B}\frac{(d_\Omega-1)^{m_\Omega}}{m_\Omega!}.$$
 \end{theorem}
Note that any partition using the vector $\delta(\{0\})=(1,0,\ldots,0)$ will not contribute to this sum, since $d_{\{0\}}-1=0$.  Letting $C$ be the set of all 
nonempty subsets of $\{1,\ldots,\ell\}$ and letting $k=(k_1,\ldots,k_\ell)$ as before, we have that  $\deg\left(ID_{r,A_{\ell,d,k}}\right)=(r+1)\delta(\delta-1)^r,$
where
\[\delta=\sum_{(m_\Omega)\in\mathcal{P}(k)}\left(1+
\sum_{\Omega\in C}m_\Omega\right)!\prod_{\Omega\in C}\frac{(d_{\Omega}-1)^{m_\Omega}}{m_\Omega!}.\]
When it is clear from context, we will abbreviate $\deg(MD_{r,A_{\ell,d,k}})$ as $\deg(MD)$ and  $\deg(ID_{r,A_{\ell,d,k}})$ as $\deg(ID)$. 

\begin{example}\label{example:111} Let us compare the degrees of the mixed and the iterated discriminant when $r=1$, $\ell=2$, and $k_1=k_2=1$.  
To compute the degree of the mixed discriminant, we consider all partitions of $\kappa=(1,1,1)$.  We may discount  any partition with the vector $(1,0,0)$, as this 
partition would contribute a term of $0$ to $\deg(MD)$.  Thus, the only relevant partitions are
\begin{itemize}
\item $(1,1,1)$,
\item $(1,1,0)+(0,0,1)$, and
\item $(1,0,1)+(0,1,0)$.
\end{itemize}
The contributions from these terms to $\deg(MD)$ are
\begin{itemize}
\item $2!\cdot \frac{(d_1+d_2)^1}{1!}=2(d_1+d_2)$ ,
\item  $3!\cdot\frac{d_1^1\cdot (d_2-1)^1}{1!\cdot 1!}=6d_1(d_2-1)$, and
\item  $3!\cdot\frac{d_2^1\cdot (d_1-1)^1}{1!\cdot 1!}=6d_2(d_1-1)$,
\end{itemize}
respectively. (Note that some of these contributions will be zero if one or both of $d_1$ and $d_2$ are equal to $1$.) Adding these gives
$$\deg(MD)=2(d_1+d_2)+6d_1(d_2-1)+6d_2(d_1-1)=12d_1d_2-4d_1-4d_2=4(3d_1d_2-d_1-d_2).$$
To compute $\deg(ID)$, we must consider the partitions of $k=(1,1)$. There are only two:  $(1,1)$ and $(1,0)+(0,1)$.  The contributions of these to $\delta$ are
\begin{itemize}
\item $2!\cdot \frac{(d_1+d_2-1)^1}{1!}=2(d_1+d_2-1)$ and
\item  $3!\cdot \frac{(d_1-1)^1(d_2-1)^1}{1!\cdot 1!}=6(d_1-1)(d_2-1),$
\end{itemize}
respectively.  Thus $\delta=2(d_1+d_2-1)+6(d_1-1)(d_2-1)=6d_1d_2-4d_1-4d_2+4.$  It follows that
$$\deg(ID)=2\delta(\delta-1)=2(6d_1d_2-4d_1-4d_2+4)(6d_1d_2-4d_1-4d_2+3).$$
We will now argue that $\deg(ID)>\deg(MD)$, unless $d_1=d_2=1$.  First we perform the change of variables $d_1=d_1'+1$ and $d_2=d_2'+1$, to remove some of the negatives.  This gives
$$\deg(MD)=4(3d_1'd_2'+2d_1'+2d_2'+1)$$
and
$$\deg(ID)=2(6d_1'd_2' + 2d_1' + 2d_2'+ 2)(6d_1'd_2' + 2d_1' + 2d_2' + 1).$$
Now, if either $d_1$ or $d_2$ is greater than $1$, then $(6d_1'd_2' + 2d_1' + 2d_2' + 1)$ is at least $3$, meaning that
$$\deg(ID)\geq 6(6d_1'd_2' + 2d_1' + 2d_2'+ 2)=36d_1'd_2' + 12d_1' + 12d_2'+ 12.$$
This is certainly greater than
$$\deg(MD)=12d_1'd_2' + 8d_1' + 8d_2'+ 4,$$
since $d_1'$ and $d_2'$ are nonnegative.  So, in this case $\deg(ID)>\deg(MD)$.  If we do have $d_1=d_2=1$, then  $\deg(MD)=4=\deg(ID)$.  
This equality was predicted by case (1) of Conjecture \ref{conjecture:main}.
\end{example}

  We will now prove that Conjecture~\ref{conjecture:main} holds in the case that $r=1$ and $d_i>1$ for all $i$.

\begin{theorem}\label{theorem:d_i>1} Suppose $d_i>1$ for all $i$.  Then the only case where $MD_{1,A_{\ell,d,k}}=ID_{1,A_{\ell,d,k}}$ is when $\ell=1$ and $d_1=2$.
\end{theorem}

\begin{proof}
This proposition holds when $\ell=1$ by Proposition \ref{proposition:r1d2}, and when $\ell=2$ and $k_1=k_2=1$ by Example \ref{example:111}.  
Thus it suffices to prove that we have $\deg\left(MD_{1,A_{\ell,d,k}}\right)<\deg\left(ID_{1,A_{\ell,d,k}}\right)$ for $\ell=2$ with $(k_1,k_2)\neq (1,1)$, and for $\ell\geq 3$.

First we consider how partitions of $\kappa=(1,k_1,\ldots,k_\ell)$ relate to partitions of $k=(k_1,\ldots,k_\ell)$.  Each partition of $\kappa$ gives rise to a 
partition of $k$ simply by deleting the first coordinate and grouping together vectors that are now identical.  Note that no partition $(m_\Omega)$ contributing 
to $\deg(MD)$ uses the vector $(1,0,\ldots,0)$.  Also, exactly one vector in each partition of $\kappa$ is of the form $(1,*,\ldots,*)$.  Call the support of this vector 
$\Psi((m_\Omega))$, or simply $\Psi$ when the context is clear.  Note that $m_{\Psi}=1$.  Let $\Xi$ denote $\Psi\setminus\{0\}$. Isolating $\Psi$ and $\Xi$, we may write
\begin{align*}
\deg\left(MD\right)=&\sum_{(m_\Omega)\in\mathcal{P}(\kappa)}\left(1+\sum_{\Omega\in B}
m_\Omega\right)!\prod_{\Omega\in B}\frac{(d_{\Omega}-1)^{m_\Omega}}{m_\Omega!}
\\=&\sum_{(m_\Omega)\in\mathcal{P}(\kappa)}\left(1+\sum_{\Omega\in B}
m_\Omega\right)!\cdot\frac{(d_{\Psi}-1)^{m_{\Psi}}}{m_{\Psi}!}\cdot\frac{(d_{\Xi}-1)^{m_{\Xi}}}{m_{\Xi}!}\prod_{\Omega\in B, \Omega\neq \Psi,\Xi}\frac{(d_{\Omega}-1)^{m_\Omega}}{m_\Omega!}
\\=&\sum_{(m_\Omega)\in\mathcal{P}(\kappa)}\left(1+\sum_{\Omega\in B}
m_\Omega\right)!\cdot(d_{\Psi}-1)\cdot\frac{(d_{\Xi}-1)^{m_{\Xi}}}{m_{\Xi}!}\prod_{\Omega\in B,\Omega\neq \Psi,\Xi}\frac{(d_{\Omega}-1)^{m_\Omega}}{m_\Omega!}.
\end{align*}
Given $(m_\Omega)$ a partition of $\kappa$, let $(n_\Omega)$ be the corresponding partition of $k$.   So, if  the term in $\deg(MD)$ coming from $(m_\Omega)$ is 
$$\left(1+\sum_{\Omega\in B}
m_\Omega\right)!\cdot(d_{\Psi}-1)\cdot\frac{(d_{\Xi}-1)^{m_{\Xi}}}{m_{\Xi}!}\prod_{\Omega\in B, \Omega\neq \Psi,\Xi}\frac{(d_{\Omega}-1)^{m_\Omega}}{m_\Omega!},$$
then the term in $\delta$ coming from $(n_\Omega)$ is 
\[\left(1+\sum_{\Omega\in B}
m_\Omega\right)!\cdot\frac{(d_{\Xi}-1)^{m_{\Xi}+1}}{(m_{\Xi}+1)!}\prod_{\Omega\in B, \Omega\neq \Psi,\Xi}\frac{(d_{\Omega}-1)^{m_\Omega}}{m_\Omega!}.\]
Note that $d_{\Xi}=d_{\Psi}-1$.  We know that $d_{\Xi}\neq 1$ by our assumption that $d_i>1$ for all $i$, so the change in factor between these two contributions is
\[\frac{d_{\Xi}-1}{d_{\Xi}(m_{\Xi}+1)}.\]
Since $d_\Xi>1$, this is at least $\frac{1}{2(m_{\Xi}+1)}$.  In general, $m_\Omega\leq \max\{k_i\}$ for any $\Omega$;  since the vector $\delta(\Psi)$ also appears in
 the partition of $\kappa$, we in fact have $m_\Xi\leq \max\{k_i\}-1$.  So,  $m_{\Xi}+1\leq \max\{k_i\}$.  It follows that $\frac{1}{2(m_{\Xi}+1)}$ is greater than or 
 equal to $\frac{1}{2\max\{k_i\}}$.  So, passing from a partition of $\kappa$ to a partition of $k$, the corresponding term in $\delta$ is at least $\frac{1}{2\max\{k_i\}}$ times the corresponding term in $\deg(MD)$.

Now we consider  how many partitions of $\kappa$ give rise to the same partition of $k$.  Given a partition $(n_\Omega)$ of $k$, all relevant partitions of 
$\kappa$ that map to it can be constructed by choosing a single vector used in $(n_\Omega)$, and appending a $1$ to the $0^{th}$ coordinate.  Thus, the number of 
partitions of $\kappa$ mapping to $(n_\Omega)$ is equal to the number of distinct vectors used in $(n_\Omega)$.  The number of distinct vectors in this partition can 
be bounded by $k_1+\ldots+k_\ell$, since this is the total sum of all the entries of all the vectors used.  Thus, we have that
$$\delta\geq\frac{1}{2\max\{k_i\}(k_1+\cdots+k_\ell)}\deg(MD).$$
It follows that
\begin{align*}
\deg(ID)=2\delta(\delta-1)\geq& 2\frac{1}{2\max\{k_i\}(k_1+\cdots+k_\ell)}\deg(MD)\cdot(\delta-1)\\=&\frac{\delta-1}{\max\{k_i\}(k_1+\cdots+k_\ell)}\cdot \deg(MD).
\end{align*}
To show that $\deg(ID)>\deg(MD)$, it remains to show that $\delta-1>\max\{k_i\}(k_1+\cdots+k_\ell)$.

First, rewrite
\begin{align*}
\delta=&\sum_{(m_\Omega)\in\mathcal{P}(k)}\left(1+
\sum_{\Omega\in C}m_\Omega\right)!\prod_{\Omega\in C}\frac{(d_{\Omega}-1)^{m_\Omega}}{m_\Omega!}
\\=&\sum_{(m_\Omega)\in\mathcal{P}(k)}\left(1+
\sum_{\Omega\in C}m_\Omega\right)\cdot\frac{\left(\sum_{\Omega\in C}m_{\Omega}\right)!}{\prod_{\Omega\in C}m_\Omega!}\prod_{\Omega\in C}(d_{\Omega}-1)^{m_\Omega}
\\=&\sum_{(m_\Omega)\in\mathcal{P}(k)}\left(1+
\sum_{\Omega\in C}m_\Omega\right)\cdot{\sum_{\Omega\in C}m_{\Omega}\choose m_{\Omega_1},\ldots,m_{\Omega_t}}\prod_{\Omega\in C}(d_{\Omega}-1)^{m_\Omega}.
\end{align*}
For any partition of $k$, we have that $\sum_{\Omega\in C}m_{\Omega}\geq \max\{k_i\}$, since $k_i$ vectors (counted with multiplicity) must have nonzero $i^{th}$ coordinate.  
Moreover, a multinomial coefficient ${a_1+a_2+\cdots+a_s\choose a_1,a_2,\cdots,a_s}$ can be rewritten as the product ${a_1\choose a_1}{a_1+a_2\choose a_2}
\cdots{a_1+a_2+\cdots+a_s\choose a_s}$, so it is at least as large as ${a_1+a_2+\cdots+a_s\choose a_s}$.  Of course, we may reorder the $a_i$'s in any way we desire.  
So, as long as some $a_i$ satisfies $0<a_i<a_1+\cdots+a_s$, we have ${a_1+a_2+\cdots+a_s\choose a_1,a_2,\cdots,a_s}\geq{a_1+a_2+\cdots+a_s\choose a_i}
\geq {a_1+a_2+\cdots+a_s\choose 1}=(a_1+\cdots+a_s)$.  This means that if $(m_\Omega)$ is a partition of $k$ that uses at least two different vectors, we have 
${\sum_{\Omega\in C}m_{\Omega}\choose m_{\Omega_1},\ldots,m_{\Omega_t}}\geq \sum_{\Omega\in C}m_\Omega\geq\max\{k_i\}$.  Finally, the product 
$\prod_{\Omega\in C}(d_{\Omega}-1)^{m_\Omega}$ is greater than or equal to $1$.  Thus, every partition of $k$ that uses at least distinct two vectors contributes at least $(1+\max\{k_i\})\cdot\max\{k_i\} $
to $\delta$.  We will now argue that there are at least $\ell$ such partitions of $k$.  

To do this, we split into two cases:  where $\ell=2$, and where $\ell\geq 3$.  If $\ell=2$ and $(k_1,k_2)\neq (1,1)$, then there are indeed at least two such partitions 
of $k=(k_1,k_2)$.  For instance, we could use $(1,1)+(k_1-1)(1,0)+(k_2-1)(0,1)$ and $k_1(1,0)+k_2(0,1)$. Both do indeed use at least two distinct vectors since at least one of $k_1-1$ and $k_2-1$ is nonzero.

Assume now $\ell\geq 3$.  We can construct a partition of $k$ that uses at least two vectors by choosing any $0-1$ vector with support size at least $2$ and at 
most $\ell-1$, and then completing the partition by using standard basis vectors.  The condition on the support size guarantees that at least one other vector will be used, 
and that this new standard basis vector has not already been used.  There are $2^\ell-2-\ell$ such initial vectors, which is greater than or equal to $\ell$ since $\ell\geq 3$.  
Thus, at least $\ell$ partitions of $k$ contribute at least $(1+\max\{k_i\})\cdot\max\{k_i\} $ to $\delta$.

Note that $\ell\max\{k_1\}\geq k_1+\cdots+k_\ell$.  It follows that
\begin{align*}
\delta\geq& \ell (1+\max\{k_i\})\cdot\max\{k_i\}
\\\geq &\ell \max\{k_i\}\cdot\max\{k_i\}+\ell
\\>&\ell \max\{k_i\}\cdot\max\{k_i\}+1
\\\geq &\max\{k_i\}(k_1+\cdots+ k_\ell)+1
\end{align*}
Equivalently, $\delta-1>\max\{k_i\}(k_1+\cdots+ k_\ell)$.  This implies that $\deg(ID)>\deg(MD)$, as desired.
\end{proof}

\section{Curves in the plane}\label{section:plane}

In this section we will determine when the mixed and iterated discriminants associated to a planar configuration are equal.  Let $A=P\cap \mathbb{Z}^2$, 
where $P$ is a smooth lattice polygon of dimension $2$.  Let $v_A$, $p_A$, and $V_A$ denote the normalized area ${\rm area}_\Z(P)$ (that is, twice its Euclidean area), 
the lattice perimeter (that is, the number of points in $A$ on the edges of $P$), and the number of vertices of $P$, respectively.  
It is well known~\cite{GKZ} that in this smooth case the degree $\delta_A$ of $D_A$ equals
$$\delta_A=3v_A-2p_A+V_A.$$

The degree of the mixed discriminant can be computed from Corollary~3.15 in~\cite{CCDRS} as
$$\deg(MD(A,A))=2({\rm area}_\Z(2P) - {\rm area}_\Z(P) -p_A)= 2 (4 v_A - v_A -p_A) = 6v_A-2p_A.$$

We can reformulate these equations in terms of the number of interior lattice points of $P$.  Let $i_A$ denote the number of interior lattice points of $P$.  Then we know by Pick's Theorem that
$$v_A=2i_A+p_A-2,$$
which can be rewritten as $v_A-p_A=2i_A-2$.
This allows us to write
$$\deg(MD(A,A))=6v_A-2p_A= 4v_A+2(v_A-p_A)=4v_A+4i_A-4=4(v_A+i_A-1).$$
and
$$\delta_A=3v_A-2p_A+V_A=v_A+2(v_A-p_A)+V_A=v_A+4(i_A-1)+V_A$$

\begin{example}  Let $A=((0,0),(2,0),(0,2))$. Let us verify that $\deg(MD(A,A))=\deg(ID_{1,A})$, as implied by Proposition \ref{proposition:r1d2}.  We have $v_A=4$, $i_A=0$, and $V_A=3$.  This gives us
$$\deg(MD(A,A))=4(v_A+i_A-1)=4(4-0-1)=12$$
and
$$\delta_A=v_A+4(i_A-1)+V_A=4+4(0-1)+3=3.$$
This means that $\deg(ID_{1,A})= 2\delta_A(\delta_A-1)=2\cdot 3\cdot 2=12=\deg(MD(A,A))$.
\end{example}

\begin{example}  Assume $A=\text{conv}((0,0),(1,0),(0,1),(1,1))$.  Let us verify that $\deg(MD(A,A))=\deg(ID_{1,A})$, as implied by Example \ref{ex:1x1x1}.   
We have $v_A=2$, $i_A=0$, and $V_A=4$.  This gives us
$$\deg(MD(A,A))=4(v_A+i_A-1)=4(2-0-1)=4$$
and
$$\delta_A=v_A+4(i_A-1)+V_A=2+4(0-1)+4=2.$$
This means that $\deg(ID_{1,A})= 2\delta_A(\delta_A-1)=2\cdot 2\cdot 1=4=\deg(MD(A,A))$.
\end{example}

It turns out that these two examples are the only smooth polygons $P$ where the iterated and the mixed discriminants associated to the
configuration of lattice points in $P$ coincide.

\begin{theorem} \label{th:plane} The only smooth polygons $P$ with an associated discriminant without singularities in 
codimension bigger than $1$ are the known cases of the triangle $2 \Delta_2$ and the unit square.
\end{theorem}

\begin{proof}  Assume $P$ is such a polygon, $A=P\cap \mathbb{Z}^2$  and $\delta_A = \deg(D_A)$.
Using our formulas for $MD(A,A)$ and $\delta_A$, we have that
$$4(v_A+i_A-1)=2(v_A+4(i_A-1)+V_A)(v_A+4(i_A-1)+V_A-1),$$
which is equivalent to
$$2(v_A+i_A-1)=(v_A+4(i_A-1)+V_A)(v_A+4(i_A-1)+V_A-1) .$$
 Suppose for the sake of contradiction that $i_A>0$. Then $v_A+i_A-1\leq v_A+4(i_A-1)<v_A+4(i_A-1)+V_A$.  Now, if $a,b,c,d$ are positive real numbers with $ab=cd$,
  then $b<c$ implies $a>d$. This means that $2>v_A+4(i_A-1)+V_A-1\geq v_A+V_A-1\geq v_A+2$.  In other words, $v_A<0$, a contradiction.  Thus we know that $i_A=0$.
 
 Setting $i_A=0$ reduces our equation to
 $$2(v_A-1)=(v_A+V_A-4)(v_A+V_A-5).$$
 By a classification result due to~\cite{Koelman} and presented again in~\cite{Castryck}, all convex lattice polygons with no interior lattice points are equivalent to 
 either the triangle $2 \Delta_2 =\text{conv}((0,0),(2,0),(0,2))$, or to a polygon of the form $\text{conv}((0,0),(0,1),(a,0)),(b,1)$, where $a\geq b\geq 0$ and $a\geq 1$. 
 These polygons are illustrated in Figure \ref{figure:polygons}.   All of these polygons have either three or four vertices.  
 So, we must have $V_A=3$ or $V_A=4$.

 \begin{figure}[hbt]
 \includegraphics{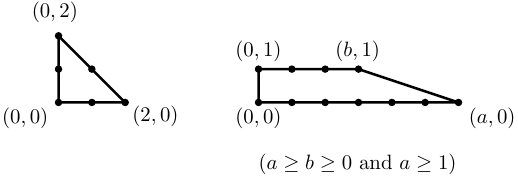}
 \caption{All lattice polygons with no interior lattice points}
 \label{figure:polygons}
 \end{figure}

If $V_A=3$, our equation becomes
 $$2(v_A-1)=(v_A-1)(v_A-2).$$
This means that either $v_A=1$, or $2=v_A-2$; that is, $v_A=1$ or $v_A=4$.  If $v_A=1$, the only possibility for $P$ is the 
primitive lattice triangle of normalized $1$; but this gives a degenerate system, and so is removed from our consideration.  
If $v_A=4$, the only possibilities of $P$ are $\text{conv}((0,0),(2,0),(0,2))$ and $\text{conv}((0,0),(4,0),(0,1))$.  The second polygon is not smooth, so the only possible triangle is $\text{conv}((0,0),(2,0),(0,2))$.

If $V_A=4$, our equation becomes
 $$2(v_A-1)=v_A(v_A-1).$$
This means that either $v_A=1$ (which is impossible impossible with $V_A=4$), or that $v_A=2$.  The only polygon with $4$ vertices and area $2$ is the square $\text{conv}((0,0),(1,0),(0,1),(1,1))$

Thus we have shown that $\text{conv}((0,0),(2,0),(0,2))$ and $\text{conv}((0,0),(1,0),(0,1),(1,1))$ are the only possibilities for $P$.  
Having already verified that $\deg(MD(A,A))=2\delta_A(\delta_A-1)$ for both these polygons, this completes the proof.
\end{proof}

\newcommand{\etalchar}[1]{$^{#1}$}

\end{document}